\documentclass[12pt, reqno]{amsart}
\usepackage{amsmath, amsthm, amscd, amsfonts, amssymb, graphicx, color, mathrsfs}
\usepackage[bookmarksnumbered, colorlinks, plainpages]{hyperref}
\usepackage[all]{xy} 
\usepackage{slashed}

\newtheorem{theorem}{Theorem}[section]
\newtheorem{lemma}[theorem]{Lemma}

\theoremstyle{definition}
\newtheorem{definition}[theorem]{Definition}
\newtheorem{example}[theorem]{Example}

\theoremstyle{remark}
\newtheorem{remark}[theorem]{Remark}
\numberwithin{equation}{section}

\begin{document}
\setcounter{page}{1}

\title[Boundedness of the dyadic maximal function on graded Lie groups ]{Boundedness of the dyadic maximal function on graded Lie groups}

\author[D. Cardona]{Duv\'an Cardona}
\address{
  Duv\'an Cardona:
  \endgraf
  Department of Mathematics: Analysis, Logic and Discrete Mathematics
  \endgraf
  Ghent University, Belgium
  \endgraf
  {\it E-mail address} {\rm duvanc306@gmail.com, duvan.cardonasanchez@ugent.be}
  }
   \author[J. Delgado]{Julio Delgado}
\address{
  Julio Delgado:
  \endgraf
  Departmento de Matematicas
  \endgraf
  Universidad del Valle
  \endgraf
  Cali-Colombia
  \endgraf
    {\it E-mail address} {\rm delgado.julio@correounivalle.edu.co}}
\author[M. Ruzhansky]{Michael Ruzhansky}
\address{
  Michael Ruzhansky:
  \endgraf
  Department of Mathematics: Analysis, Logic and Discrete Mathematics
  \endgraf
  Ghent University, Belgium
  \endgraf
 and
  \endgraf
  School of Mathematical Sciences
  \endgraf
  Queen Mary University of London
  \endgraf
  United Kingdom
  \endgraf
  {\it E-mail address} {\rm michael.ruzhansky@ugent.be, m.ruzhansky@qmul.ac.uk}
  }

\thanks{The authors are supported  by the FWO  Odysseus  1  grant  G.0H94.18N:  Analysis  and  Partial Differential Equations and by the Methusalem programme of the Ghent University Special Research Fund (BOF)
(Grant number 01M01021). J. Delgado is also supported by Vice. Inv.Universidad del Valle Grant CI 71329, MathAmSud and Minciencias-Colombia under the project MATHAMSUD21-MATH-03. Duv\'an Cardona was supported by the Research Foundation-Flanders
(FWO) under the postdoctoral
grant No 1204824N.  Michael Ruzhansky is also supported  by EPSRC grant 
EP/R003025/2.
}


\begin{abstract} Let  $1<p\leq  \infty$ and let $n\geq 2.$ It was proved independently by C. Calder\'on, R. Coifman and G. Weiss that the dyadic maximal function
\begin{equation*}
    \mathcal{M}^{d\sigma}_Df(x)=\sup_{j\in\mathbb{Z}}\left|\smallint\limits_{\mathbb{S}^{n-1}}f(x-2^jy)d\sigma(y)\right|
\end{equation*} is a bounded operator on $L^p(\mathbb{R}^n)$  where $d\sigma(y)$ is the surface measure on $\mathbb{S}^{n-1}.$  In this paper we prove an analogue of  this result on arbitrary graded Lie groups. More precisely, to any finite Borel measure  $d\sigma$ with compact support on a graded Lie group $G,$ we associate the corresponding dyadic maximal function $\mathcal{M}_D^{d\sigma}$ using the homogeneous structure of the group. Then, we prove  a criterion in terms of the order (at zero and at infinity) of the group Fourier transform $\widehat{d\sigma}$ of $d\sigma$ with respect to a fixed Rockland operator $\mathcal{R}$ on $G$ that assures the boundedness of $\mathcal{M}_D^{d\sigma}$ on  $L^p(G)$ for all $1<p\leq \infty.$   
\end{abstract} 

\maketitle

\tableofcontents
\allowdisplaybreaks

\section{Introduction}

\subsection{Outline}
For more than fifty years (dating back to the work of Folland and Stein \cite{FollandStein1982}
in the 1970s), there has been an extensive  program to generalise the techniques
from the real-variable Euclidean  harmonic analysis  to the more general setting
of the nilpotent Lie groups. This  is particularly motivated by applications
to degenerate partial differential operators (see e.g. Rothschild and Stein \cite{RothschildStein76}), where the usual methods
on the Euclidean space are not completely suitable. Among the fundamental operators of the Euclidean harmonic analysis, are the full maximal function and its dyadic counterpart. Contributing to the aforementioned program, the aim of this work is to study the $L^p$-boundedness of the dyadic maximal function on a nilpotent Lie group $G.$ Since our criteria involve that the group admits left-invariant hypoelliptic partial differential operators we assume that $G$ is a graded Lie group in view of the Helffer and Nourrigat solution of the Rockland conjecture, see \cite{HelfferNourrigat}.    

To the best of our knowledge, the study of the $L^p$-boundedness of the spherical averages started in a satisfactory way  with Stein (inspired by the 1976's work of  Nagel,  Rivi\`ere and  Wainger \cite{Nagel1976}). Indeed, consider the full-maximal function on $\mathbb{R}^n$
\begin{equation}\label{Full}
    \mathcal{M}^{d\sigma}_{F}f(x)=\sup_{r>0}\left|\smallint\limits_{\mathbb{S}^{n-1}}f(x-ry)d\sigma(y)\right|,
\end{equation}where $\sigma$ is the surface measure on the sphere $\mathbb{S}^{n-1}.$ A remarkable result due to Stein in \cite{Stein1976}, proved the boundedness of $\mathcal{M}^{d\sigma}_{F}$ from  $L^p(\mathbb{R}^n)$ to itself, if and only if $p>\frac{n}{n-1},$ for all $n\geq 3.$ Then, the lower dimensional case $n=2$ was proved by Bourgain in \cite{Bourgain1986}. 
Additionally,   C.  Calder\'on in \cite{Calderon1979}, proved that the dyadic maximal function
\begin{equation}\label{Dyadic:rn}
   \mathcal{M}^{d\sigma}_Df(x)=\sup_{j\in \mathbb{Z}}\left|\smallint\limits_{\mathbb{S}^{n-1}}f(x-2^jy)d\sigma(y)\right|,
\end{equation}can be extended to a bounded operator on $L^p(\mathbb{R}^n),$ for all $1<p\leq \infty,$ for any $n\geq 2.$  It was observed by S. Wainger  that the $L^p$-boundedness of the dyadic maximal function  was also proved  independently by Coifman and Weiss in \cite{CoifmanWeiss1978}. Other proofs for the $L^p$-boundedness of \eqref{Full} when $n\geq 3$ can be found e.g. in Carbery \cite{Carbery},  Cowling and Mauceri \cite{CowlingMaceuri}, Rubio de Francia \cite{RubioDeFrancia},   Duoandikoetxea and Rubio de Francia \cite{Duoandikoetxea:RubioDeFrancia:86}, and for Bourgain's result  for $n=2$ a new proof was given by Mockenhaupt,  Seeger, and  Sogge \cite{Mockenhaupt1992}.

The study of the full maximal function and its dyadic version has been mainly concentrated in the context of the Heisenberg group $\mathbb{H}_n$ and on two-steps nilpotent Lie groups. More precisely, by considering the unit sphere $\mathbb{S}_{\mathbb{H}_n,K}$ with respect to the Kor\'anyi norm $|(z,t)|=(|z|^4+16t^2)^{\frac{1}{4}}$ there is a unique Rad\'on measure $d\sigma$ that makes valid the polar decomposition formula. Then, Cowling in \cite{Cowling79}  proved the $L^p$-boundedness of the corresponding full-maximal function $\mathcal{M}^{d\sigma}_F.$  Extensions of Cowling's result have been obtained e.g. by Schmidt in \cite{Schmidt}   to hypersurfaces with non-vanishing rotational curvature on nilpotent Lie groups and for two steps nilpotent Lie groups by Fischer \cite{Fischer2006}.
The $L^p$-boundedness for the Lacunary  maximal function in this setting has been proved recently by Ganguly and  Thangavelu \cite{GangulyThangavelu2021} on $\mathbb{H}_n$ for all $n\geq 2.$ The approach in \cite{GangulyThangavelu2021} combines the ideas of Cowling \cite{Cowling79} with the sparse techniques developed (in the Euclidean setting) by Lacey in \cite{Lacey2019}. As for the spherical maximal function on the Heisenberg group for the surface measure on the complex sphere $\mathbb{S}_{\mathbb{H}_n,r}:=\{(z,0):|z|=r\},$ we refer the reader to Nevo and Thangavelu \cite{NevoThangavelu}. The result in \cite{NevoThangavelu} was improved independently by Narayanan and  Thangavelu \cite{NarayananThangavelu2004} and by  M\"uller and Seeger \cite{MullerSeeger2004}.  Moreover, in \cite{BagchiHaitRoncalThangavelu2018}, Bagchi, Hait, Roncal and Thangavelu have  proved an analogue of Calder\'on’s theorem for the associated lacunary spherical maximal function on $\mathbb{H}_n$ with $n\geq 2$. For $n=1$ the $L^p$-boundendess of the corresponding full maximal function (restricted to a class of radial functions) has been proved by Beltr\'an, Guo, Hickman and Seeger in \cite{BeltranGuoHickmanSeeger}.

\subsection{The main result}

Let $G$ be a homogeneous Lie group. Let us consider its corresponding family of dilations (see Definition \ref{dilations:group})
$$ D_{r}:G\rightarrow G,\,\,x\mapsto D_{r}(x)\equiv r\cdot x,\,x\in G.$$  Consider the dyadic maximal function
\begin{equation}\label{Maximal:Function:Graded}
    \mathcal{M}^{d\sigma}_Df(x)=\sup_{j\in \mathbb{Z}}\left|\smallint\limits_{G}f\left(x\cdot (2^j\cdot y)^{-1}\right)d\sigma(y)\right|,
\end{equation} associated to an arbitrary finite Borel measure $d\sigma$   with compact support on $G.$ We require the existence of (Rockland operators which are) left-invariant hypoelliptic partial differential operators on the group and then the group has to be graded (see \cite[Page 172]{FischerRuzhanskyBook} and Definition \ref{graded:Lie:groups}).  Then, a positive Rockland operator  $\mathcal{R}$ is a group Fourier multiplier by the operator-valued function $\pi(\mathcal{R}),$ (that is the infinitesimal representation of the operator,
defined at any irreducible and unitary representation $\pi$ of the unitary dual $\widehat{G}$ of $G$ as in \eqref{Symbol:R}). In order to analyse the $L^p$-boundedness of the Dyadic maximal function \eqref{Maximal:Function:Graded} we will assume that for some $a>0,$ the measure $d\sigma$ in \eqref{Maximal:Function:Graded} satisfies the group  Fourier transform condition
\begin{equation}\label{FT;Condition:Measure:Intro}
\max_{\pm}\sup_{\pi\in\widehat{G}}\Vert\pi(\mathcal{R})^{\pm \frac{a}{\nu}}\widehat{d\sigma}(\pi)\Vert_{\textnormal{op}}<\infty,
\end{equation}where $\nu$ is the homogeneous degree of the operator $\mathcal{R},$ (in the case of the positive Laplacian $-\Delta,$ $\nu=2$). Then, according to the discussion above,  \eqref{FT;Condition:Measure:Intro} says that $\widehat{d\sigma}(\pi)$ has order  $+a$ at infinity and $-a$ at zero (with respect to the spectrum of any $\pi(\mathcal{R})$).

The following criterion is the main theorem of this work.
\begin{theorem}\label{Main:Thn} Let $d\sigma$ be a finite Borel measure of compact support on a graded Lie group $G.$ Let $\mathcal{R}$ be a positive Rockland operator on $G$ of homogeneous degree $\nu>0.$ Assume that for some $a>0$ the Fourier transform of $d\sigma$ satisfies the growth estimate
\begin{equation}\label{FT;Condition:Measure}
\max_{\pm}  \sup_{\pi\in\widehat{G}}\Vert\pi(\mathcal{R})^{\pm \frac{a}{\nu}}\widehat{d\sigma}(\pi)\Vert_{\textnormal{op}}<\infty.
\end{equation}
Then, the dyadic maximal function
$\mathcal{M}^{d\sigma}_D:L^p(G)\rightarrow L^p(G)$ can be extended to a bounded operator for all $1<p\leq \infty.$
\end{theorem}
\begin{remark}
For the $\ell^p$-boundeness of the discrete dyadic maximal function we refer the reader to Bourgain, Mirek,  Stein, and Wr\'obel \cite{Bourgain2020,Bourgain2021}. 
\end{remark}
\begin{remark}
\label{remark:sigma} Our proof of Theorem \ref{Main:Thn}  is inspired by the approach developed by Duoandikoetxea and Rubio De Francia in \cite{Duoandikoetxea:RubioDeFrancia:86}. In the setting of non-commutative nilpotent Lie groups many difficulties arise. One is that the Fourier transform of distributions is operator valued. Also, the Fourier transform condition \eqref{FT;Condition:Measure} is motivated by conditions of the same nature arising for homogeneous structures (nonisotropic structures) on the Euclidean space, namely,  the condition of non-vanishing curvature  at infinite order for Euclidean hypersurfaces $\Sigma\subset \mathbb{R}^n.$ Such a condition   implies a  decay estimate for the Fourier transform $\widehat{d\Sigma}$ of the corresponding surface measure $d\Sigma,$  see e.g. Seeger, Tao and Wright \cite{SeegerTaoWright}.
\end{remark}
\begin{remark}\label{rem:heisenberg} The criterion given in Theorem \ref{Main:Thn} is new even on the Heisenberg group. Then also on Heisenberg type groups, stratified groups, etc. Examples of Rockland operators on stratified groups are H\"ormander sub-Laplacians and their integer powers, see \cite{FischerRuzhanskyBook}.
\end{remark}
\begin{remark}Let $d\sigma$ be a finite measure of compact support on $G=\mathbb{R}^n.$  In the case where  $\mathcal{R}=-\Delta_x$ is  the positive Laplacian, the inequality in \eqref{FT;Condition:Measure}
  becomes equivalent to the Fourier transform condition, 
  \begin{equation}\label{Euclidean:cond}
   \forall\xi\neq 0,\,  |\widehat{d\sigma}(\xi)|\lesssim \min\{|\xi|^{a},|\xi|^{-a}\},\,\,a>0,
  \end{equation}  to guarantee the $L^p$-boundedness of the dyadic maximal operator $\mathcal{M}_D^{d\sigma}$ on $L^p(\mathbb{R}^n),$ for all $1<p\leq \infty,$ see Duoandikoetxea and Rubio De Francia \cite{Duoandikoetxea:RubioDeFrancia:86}, and e.g.  Theorem 6.3.4 in Grafakos \cite[Page 455]{Grafakos}. Let $d\mu$ be the surface measure on the sphere $\mathbb{S}^{n-1}.$ Coifman and Weiss \cite[Page 246]{CoifmanWeiss1978} deduced the $L^p$-boundedness of $\mathcal{M}_D^{d\mu},$ $1\leq p<\infty,$  using the fact that
  \begin{equation}\label{Euclidean:cond:2}
   \forall\xi\neq 0,\,  |\widehat{d\mu}(\xi)|\lesssim|\xi|^{-(n-1)/2},\,\,n\geq 2.
  \end{equation}   Indeed,   if $\phi\in C^\infty_0(\mathbb{R}^n)$ is such that $\widehat{\phi}(0)=1,$ for  the measure $d{\sigma}=d\mu-\widehat{\mu}(0)\phi,$  Coifman and Weiss proved that the $L^p$-boundedness of  $\mathcal{M}_D^{d\sigma},$ implies the $L^p$-boundedness of  $\mathcal{M}_D^{d\mu}.$  This latter argument due to Coifman and Weiss is an alternative proof to the one given in the classical manuscript \cite{Calderon1979} due to  C.  Calder\'on. 
  
\end{remark}
 \begin{remark}Let $d\sigma$ be a finite Borel measure    of compact support on a graded Lie group $G.$ Let $\mathcal{R}$ be a positive Rockland operator on $G$ of homogeneous degree $\nu>0.$ By the Riesz-representation theorem, we have that $d\sigma=Kdx$ for some compactly supported function $K$ on $G.$ Note that if $\mathcal{R}^{\pm \frac{a}{\nu}}K\in L^1(G),$ (this means that $K$ belongs to the Sobolev space $\dot{L}^{1}_{\pm a}(G),$ that is $\Vert K\Vert_{\dot{L}^{1}_{\pm a}(G)}:=\Vert \mathcal{R}^{\pm \frac{a}{\nu} }K\Vert_{L^{1}(G)}<\infty$) then 
    $$  \sup_{\pi\in\widehat{G}}\Vert\pi(\mathcal{R})^{\pm \frac{a}{\nu}}\widehat{d\sigma}(\pi)\Vert_{\textnormal{op}}= \sup_{\pi\in\widehat{G}}\Vert\smallint_{G}\mathcal{R}^{\pm \frac{a}{\nu} }K(g)\pi(g)^{*}dg\Vert_{\textnormal{op}}\leq  \Vert K\Vert_{\dot{L}^{1}_{\pm a}(G)} <\infty, $$ showing that the class of compactly supported finite Borel measures $d\sigma=Kdx$ with $K\in \dot{L}^{1}_{a}(G)\cap \dot{L}^{1}_{-a}(G) , $ for some $a>0,$ provides examples of  measures satisfying \eqref{FT;Condition:Measure}. 
\end{remark}
\begin{remark}
    We note that the aforementioned result by C. P. Calder\'on \cite{Calderon1979} about the $L^p$-boundedness of \eqref{Dyadic:rn} can be substantially extended in the Euclidean setting. Indeed, as in Remark \ref{remark:sigma}, let   $\Sigma\subset \mathbb{R}^n$  be a compact hyper-surface. If the Gaussian curvature does not vanish to infinite order on  $\Sigma,$ then one has the decay estimate  $|\widehat{d\Sigma}(\xi)|\lesssim (1+|\xi|)^{-\gamma},$  for some $\gamma>0.$ This is a consequence of van der Corput’s lemma, see e.g. Seeger, Tao, and Wright \cite[Page 622]{SeegerTaoWright}. However, due to the generality of the setting of Theorem  \ref{Main:Thn}, it is unknown to us if a similar argument could prove the necessity of \eqref{FT;Condition:Measure} on arbitrary graded Lie groups.
\end{remark}

\begin{remark} Now we continue with Remark \ref{rem:heisenberg}. In the Heisenberg group $\mathbb{H}^n$, our results are new in the sense that the criterion in \eqref{FT;Condition:Measure:Intro} allows the use of general Rockland operators $\mathcal{R}$ of higher homogeneous orders.  The sub-Laplacian $\mathcal{L}=-\sum_{j=1}^n(X_j^2+Y_j^2)$ is a Rockland operator of homogeneous order $\nu=2.$ One can consider for instance the positive Rockland operator $\mathcal{R}=\sum_{j=1}^n(X_j^4+Y_j^4).$ Here $\{X_j,Y_j,Z=\partial_{t}:1\leq j\leq n\}$ denotes the standard basis of left-invariant vector fields of the Lie algebra $\mathfrak{h}_n=\textnormal{Lie}(\mathbb{H}^n),$ see e.g. \cite[Chapter 6]{FischerRuzhanskyBook} for details.   It would be interesting to analyse the relation between the condition in \eqref{FT;Condition:Measure} and other results about the $L^p$-boundedness of $\mathcal{M}^{d\sigma}_D$ on the Heisenberg group, see  \cite{BagchiHaitRoncalThangavelu2018}, \cite{BeltranGuoHickmanSeeger}, \cite{Cowling79}, \cite{Schmidt}, \cite{GangulyThangavelu2021}, \cite{NevoThangavelu},  \cite{NarayananThangavelu2004},  and \cite{RoosSeegerSrivastava22}. We also refer the reader to the recent work \cite{GovidanAswinHickman}  where a {\it curvature assumption} was analysed to guarantee the $L^p$-boundedness of the dyadic maximal function on homogeneous Lie groups. 
\end{remark}
\begin{remark}Although the family of graded Lie groups is a sub-class of the family of homogeneous Lie groups (and every homogeneous Lie group is a nilpotent Lie group), every homogeneous Lie group admitting the existence of positive hypoelliptic partial differential operators (also called Rockland operators) is necessarily a graded Lie group in view of the Helffer and Nourrigat solution of the Rockland conjecture, see \cite{HelfferNourrigat}.  So, it is an open problem to determine a relation between the {\it curvature assumption} in \cite{GovidanAswinHickman} and the group Fourier transform decay property in \eqref{FT;Condition:Measure} when $G$ is a graded Lie group, or even when the group $G$ is stratified and the Rockland operator $\mathcal{R}$ is considered to be the sub-Laplacian.
\end{remark}

\begin{remark} We observe that even in the Euclidean case, the problem of determining the weak (1,1) boundedness of the dyadic spherical means is open. Some results in this direction have been proved by Christ and Stein\cite{ChristStein87},  Christ \cite{Christ88},   Seeger, Tao, and  Wright \cite{STW}, and Cladek and Krause \cite{CladekKrause}. In consequence,  it
is a very interesting open problem to determine whether  \eqref{FT;Condition:Measure} is a sufficient condition to guarantee the weak (1,1) boundedness of   $\mathcal{M}^{d\sigma}_D$. 
\end{remark}

\section{Fourier analysis on graded groups}\label{preliminaries}

For the aspects of the Fourier analysis on nilpotent Lie groups we follow Folland and Stein \cite{FollandStein1982} and the notation is taken from \cite{FischerRuzhanskyBook}. For the aspects about the theory of Rockland operators on graded Lie groups we follow \cite{FischerRuzhanskyBook}.

\subsection{Homogeneous and graded Lie groups} 
    Let $G$ be a homogeneous Lie group, that is a connected and simply connected Lie group whose Lie algebra $\mathfrak{g}$ is endowed with a family of dilations $D_{r,\mathfrak{g}}.$ We define it  as follows.
    \begin{definition} A family of dilations $  \textnormal{Dil}(\mathfrak{g}):=\{D_{r,\mathfrak{g}}:\,r>0\}$  on the Lie algebra $\mathfrak{g}$ is a family of  automorphisms on $\mathfrak{g}$  satisfying the following two compatibility conditions:
\begin{itemize}
\item[1.] For every $r>0,$ $D_{r,\mathfrak{g}}$ is a map of the form
$ D_{r,\mathfrak{g}}=\textnormal{Exp}(\ln(r)A), $ 
for some diagonalisable linear operator $A\equiv \textnormal{diag}[\nu_1,\cdots,\nu_n]:\mathfrak{g}\rightarrow \mathfrak{g}.$  
\item[2.] $\forall X,Y\in \mathfrak{g}, $ and $r>0,$ $[D_{r,\mathfrak{g}}X, D_{r,\mathfrak{g}}Y]=D_{r,\mathfrak{g}}[X,Y].$ 
\end{itemize}
\end{definition}
\begin{remark}
We call  the eigenvalues of $A,$ $\nu_1,\nu_2,\cdots,\nu_n,$ the dilations weights or weights of $G$. 
\end{remark} 
In our analysis  the notion of the homogeneous dimension of the group is crucial. We introduce it as follows.
\begin{definition}
The homogeneous dimension of a homogeneous Lie group $G$ whose dilations are defined via $ D_{r,\mathfrak{g}}=\textnormal{Exp}(\ln(r)A),$ is given by  $  Q=\textnormal{Tr}(A)=\nu_1+\cdots+\nu_n,  $  where $\nu_{i},$ $i=1,2,\cdots, n,$ are the eigenvalues of $A.$
\end{definition}
In terms of the weights of the group $G$ we define for every $\alpha\in \mathbb{N}_0^n,$ the function
\begin{equation}\label{weighted:lenght}
    [\alpha]:=\sum_{j=1}^n\nu_j\alpha_j.
\end{equation}
\begin{definition}[Dilations on the group]\label{dilations:group}
The family of dilations $\textnormal{Dil}(\mathfrak{g})$ of the Lie algebra $\mathfrak{g}$ induces a family of  mappings on $G$ defined via,
$$\textnormal{Dil}(G):=\{ D_{r}:=\exp_{G}\circ D_{r,\mathfrak{g}} \circ \exp_{G}^{-1}:\,\, r>0\}, $$
where $\exp_{G}:\mathfrak{g}\rightarrow G$ is the usual exponential mapping associated to the Lie group $G.$ We refer to the elements of the family $\textnormal{Dil}(G)$ as dilations on the group.
\end{definition} 
 \begin{remark}
 If we use the notation $r\cdot x=D_{r}(x),$ $x\in G,$ $r>0,$ then the effect of the dilations of the group  on the Haar measure $dx$ on $G$ is determined by the identity $$ \smallint\limits_{G}(f\circ D_{r})(x)dx=r^{-Q}\smallint\limits_{G}f(x)dx. $$
 \end{remark}
 \begin{definition}\label{graded:Lie:groups}
A connected, simply connected Lie group $G$ is graded if its Lie algebra $\mathfrak{g}$ may be decomposed as the direct sum of subspaces $  \mathfrak{g}=\mathfrak{g}_{1}\oplus\mathfrak{g}_{2}\oplus \cdots \oplus \mathfrak{g}_{s}$  such that the following bracket conditions are satisfied: 
$[\mathfrak{g}_{i},\mathfrak{g}_{j} ]\subset \mathfrak{g}_{i+j},$ where  $ \mathfrak{g}_{i+j}=\{0\}$ if $i+j\geq s+1,$ for some $s.$     
 \end{definition}
  Examples of graded Lie groups are the Heisenberg group $\mathbb{H}^n$ and more generally any stratified group where the Lie algebra $ \mathfrak{g}$ is generated by the first stratum $\mathfrak{g}_{1}$.  Here, $n$ is the topological dimension of $G,$ $n=n_{1}+\cdots +n_{s},$ where $n_{k}=\mbox{dim}\mathfrak{g}_{k}.$ For more examples, see \cite{FischerRuzhanskyBook}.
\begin{remark}[Not every
nilpotent Lie group is homogeneous]
A Lie algebra admitting a family of dilations is nilpotent, and hence so is its associated
connected, simply connected Lie group. The converse does not hold, i.e., not every
nilpotent Lie group is homogeneous  although they exhaust a large class, see \cite{FischerRuzhanskyBook} for details. Indeed, the main class of Lie groups under our consideration is that of graded Lie groups. \end{remark}

\subsection{Fourier analysis on nilpotent Lie groups}

Let $G$ be a simply connected nilpotent Lie group. Then the adjoint representation $\textnormal{ad}:\mathfrak{g}\rightarrow\textnormal{End}(\mathfrak{g})$ is nilpotent. Next, we define unitary and irreducible representations.
\begin{definition}[Continuous, unitary and irreducible representations of $G$]We say that $\pi$ is a continuous, unitary and irreducible  representation of $G,$ if the following properties are satisfied,
\begin{itemize}
    \item[1.] $\pi\in \textnormal{Hom}(G, \textnormal{U}(H_{\pi})),$ for some separable Hilbert space $H_\pi,$ i.e. $\pi(xy)=\pi(x)\pi(y)$ and for the  adjoint of $\pi(x),$ $\pi(x)^*=\pi(x^{-1}),$ for every $x,y\in G.$ This property says that the representation is compatible with the group operation.
    \item[2.] The map $(x,v)\mapsto \pi(x)v, $ from $G\times H_\pi$ into $H_\pi$ is continuous. This says that the representation is a strongly continuous mapping.
    \item[3.] For every $x\in G,$ and $W_\pi\subset H_\pi,$ if $\pi(x)W_{\pi}\subset W_{\pi},$ then $W_\pi=H_\pi$ or $W_\pi=\{0\}.$ This means that  the representation $\pi$ is irreducible if its only invariant subspaces are $W=\{0\}$ and $W=H_\pi,$ the trivial ones. 
\end{itemize}
\end{definition}
\begin{definition}[Equivalent representations]
    Two unitary representations $$ \pi\in \textnormal{Hom}(G,\textnormal{U}(H_\pi)) \textnormal{ and  }\eta\in \textnormal{Hom}(G,\textnormal{U}(H_\eta))$$  are equivalent if there exists a bounded linear mapping $Z:H_\pi\rightarrow H_\eta$ such that for any $x\in G,$ $Z\pi(x)=\eta(x)Z.$ The mapping $Z$ is called an intertwining operator between $\pi$ and $\eta.$ The set of all the intertwining operators between $\pi$ and $\eta$ is denoted by $\textnormal{Hom}(\pi,\eta).$
\end{definition}
\begin{definition}[The unitary dual]
    The relation $\sim$ on the set of unitary and irreducible representations $\textnormal{Rep}(G)$ defined by: {\it $\pi\sim \eta$ if and only if $\pi$ and $\eta$ are equivalent representations,} is an equivalence relation. The quotient 
$$
    \widehat{G}:={\textnormal{Rep}(G)}/{\sim}
$$is called the unitary dual of $G.$
\end{definition}
The unitary dual encodes all the Fourier analysis on the group. So, we are going to define the Fourier transform.
\begin{definition}[Group Fourier Transform]
The Fourier transform of $f\in L^1(G), $ at $\pi\in\widehat{G},$ is defined by 
\begin{equation*}
    \widehat{f}(\pi)=\smallint\limits_{G}f(x)\pi(x)^*dx:H_\pi\rightarrow H_\pi.
\end{equation*}
\begin{remark}
    The Schwartz space $\mathscr{S}(G)$ is defined by the smooth functions $f:G\rightarrow\mathbb{C},$ such that via the exponential mapping $f\circ \exp_{G}:\mathfrak{g}\cong \mathbb{R}^n\rightarrow \mathbb{C}$ can be identified with functions on the Schwartz class $\mathscr{S}(\mathbb{R}^n).$ Then, the Schwartz space on the dual $\widehat{G}$ is defined by the image under the Fourier transform of the Schwartz space $\mathscr{S}(G),$ that is  $\mathscr{F}_{G}:\mathscr{S}(G)\rightarrow \mathscr{S}(\widehat{G}):=\mathscr{F}_{G}(\mathscr{S}(G)).$ 
\end{remark}
\end{definition}
\begin{remark}[Fourier Inversion Formula and Plancherel Theorem]
If we identify one representation $\pi$ with its equivalence class, $[\pi]=\{\pi':\pi\sim \pi'\}$,  for every $\pi\in \widehat{G}, $ the Kirillov trace character $\Theta_\pi$ defined by  $$[\Theta_{\pi},f]:
=\textnormal{Tr}(\widehat{f}(\pi)),$$ is a tempered distribution on $\mathscr{S}(G).$ The Kirillov character allows one to write the Fourier inversion formula
\begin{equation*}
  \forall f\in L^1(G)\cap L^2(G),\,\, f(x)=\smallint\limits_{\widehat{G}}\textnormal{Tr}[\pi(x)\widehat{f}(\pi)]d\pi,\,\,x\in G.
\end{equation*}
The $L^2$-space on the dual is defined by the completion of $\mathscr{S}(G)$ with respect to the norm
\begin{equation}
    \Vert \sigma\Vert_{L^2(\widehat{G})}:=\left(\smallint\limits_{\widehat{G}}\|\sigma(\pi)\|_{\textnormal{HS}}^2d\pi \right)^{\frac{1}{2}},\,\,\sigma(\pi)=\widehat{f}(\pi)\in \mathscr{S}(\widehat{G}),
\end{equation}where $\|\cdot\|_{\textnormal{HS}}$ denotes the Hilbert-Schmidt norm of  operators on every representation space. The corresponding inner product on $L^2(\widehat{G})$ is given by
\begin{equation}
    (\sigma,\tau)_{L^2(\widehat{G})}:=\smallint\limits_{G}\textnormal{Tr}[\sigma(\pi)\tau(\pi)^{*}]d\pi,\,\,\sigma,\tau\in L^2(\widehat{G}),
\end{equation}where the notation $\tau(\pi)^{*}$ indicates the adjoint operator. Then,
the Plancherel theorem says that $\Vert f\Vert_{L^2(G)}=\Vert \widehat{f}\Vert_{L^2(\widehat{G})}$ for all $f\in L^2(G).$
\end{remark}

\subsection{Homogeneous linear operators and Rockland operators} Homogeneous operators interact with the dilations of the group. We introduce them in the following definition.
\begin{definition}[Homogeneous operators]
A continuous linear operator $T:C^\infty(G)\rightarrow \mathscr{D}'(G)$ is homogeneous of  degree $\nu_T\in \mathbb{C}$ if for every $r>0$ the equality 
\begin{equation*}
T(f\circ D_{r})=r^{\nu_T}(Tf)\circ D_{r}
\end{equation*}
holds for every $f\in \mathscr{D}(G). $
\end{definition}
Now, using the notation in \eqref{weighted:lenght} we introduce the main class of differential operators in the context of nilpotent Lie groups. The existence of these operators classifies the family of graded Lie groups. We call them Rockland operators.
\begin{definition}[Rockland operators]
If for every representation $\pi\in\widehat{G},$ $\pi:G\rightarrow U({H}_{\pi}),$ we denote by ${H}_{\pi}^{\infty}$ the set of smooth vectors (also called G\r{a}rding vectors), that is, the space of vectors $v\in {H}_{\pi}$ such that the function $x\mapsto \pi(x)v,$ $x\in \widehat{G},$ is smooth,  a Rockland operator is a left-invariant partial  differential operator $$ \mathcal{R}=\sum_{[\alpha]=\nu}a_{\alpha}X^{\alpha}:C^\infty(G)\rightarrow C^{\infty}(G)$$  which is homogeneous of positive degree $\nu=\nu_{\mathcal{R}}$ and such that, for every unitary irreducible non-trivial representation $\pi\in \widehat{G},$ its symbol $\pi(\mathcal{R})$ defined via the Fourier inversion formula by
\begin{equation}\label{Symbol:R}
   \mathcal{R}f(x)= \smallint\limits_{\widehat{G}}\textnormal{Tr}[\pi(x)\pi(\mathcal{R})\widehat{f}(\pi)]d\pi,\,\,x\in G,
\end{equation}
is injective on ${H}_{\pi}^{\infty};$ $\sigma_{\mathcal{R}}(\pi)=\pi(\mathcal{R})$ coincides with the infinitesimal representation of $\mathcal{R}$ as an element of the universal enveloping algebra $\mathfrak{U}(\mathfrak{g})$. 
\end{definition}

\begin{example}
Let $G$ be a graded Lie group of topological dimension $n.$ We denote by $\{D_{r}\}_{r>0}$ the natural family of dilations of its Lie algebra $\mathfrak{g}:=\textnormal{Lie}(G),$ and by $\nu_1,\cdots,\nu_n$ its weights.  We fix a basis $Y=\{X_1,\cdots, X_{n}\}$ of  $\mathfrak{g}$ satisfying $D_{r}X_j=r^{\nu_j}X_{j},$ for $1\leq j\leq n,$ and all $r>0.$ If $\nu_{\circ}$ is any common multiple of $\nu_1,\cdots,\nu_n,$ the  operator 
$$ \mathcal{R}=\sum_{j=1}^{n}(-1)^{\frac{\nu_{\circ}}{\nu_j}}c_jX_{j}^{\frac{2\nu_{\circ}}{\nu_j}},\,\,c_j>0, $$ is a positive Rockland operator of homogeneous degree $2\nu_{\circ }$ on $G$ (see Lemma 4.1.8 of \cite{FischerRuzhanskyBook}).    
\end{example}
\begin{remark}
It can be shown that a Lie group $G$ is graded if and only if there exists a differential Rockland operator on $G.$ Also, in view of the Schwartz kernel theorem, the operator $\mathcal{R}$ admits a right-convolution kernel $k_{\mathcal{R}},$ that is for any $f\in C^\infty_0(G),$
\begin{equation*}
    \mathcal{R}f(x)=\smallint\limits_{G}f(y)k_{\mathcal{R}}(y^{-1}x)dy,\,\,x\in G.
\end{equation*}In terms of the convolution of functions
\begin{equation}
    f\ast g(x):=\smallint\limits_{G}f(y)g(y^{-1}x)dy,\,\,f,g\in L^1(G),
\end{equation}and in view of the action of the Fourier transform on convolutions $  \widehat{f\ast g}=\widehat{g}\widehat{f},$ the Fourier inversion formula shows that 
$ 
    \forall \pi \in \widehat{G},\,\pi(\mathcal{R})=\widehat{k}_{\mathcal{R}}(\pi).
$
\end{remark}
Next, we record for our further analysis some aspects of the functional calculus for Rockland operators.
\begin{remark}[Functional calculus for Rockland operators] If the Rockland operator is formally self-adjoint, then $\mathcal{R}$ and $\pi(\mathcal{R})$ admit self-adjoint extensions on $L^{2}(G)$ and ${H}_{\pi},$ respectively.
Now if we preserve the same notation for their self-adjoint
extensions and we denote by $E$ and $E_{\pi}$  their spectral measures, we will denote by
\begin{equation}\label{Functional:identities:spectral:calculus}
    \psi(\mathcal{R})=\smallint\limits_{-\infty}^{\infty}\psi(\lambda) dE(\lambda),\,\,\,\textnormal{and}\,\,\,\pi(\psi(\mathcal{R}))\equiv \psi(\pi(\mathcal{R}))=\smallint\limits_{-\infty}^{\infty}\psi(\lambda) dE_{\pi}(\lambda),
\end{equation}
the functions defined by the functional calculus. 
In general, we will reserve the notation $\{dE_A(\lambda)\}_{0\leq\lambda<\infty}$ for the spectral measure associated with a positive and self-adjoint operator $A$ on a Hilbert space $H.$ 
\end{remark}
We now recall a lemma on dilations on the unitary dual $\widehat{G},$ which will be useful in our analysis of spectral multipliers.   For the proof, see Lemma 4.3 of \cite{FischerRuzhanskyBook}.
\begin{lemma}\label{dilationsrepre}
For every $\pi\in \widehat{G}$ let us define  
\begin{equation}\label{dilations:repre}
  D_{r}(\pi)(x)\equiv (r\cdot \pi)(x):=\pi(r\cdot x)\equiv \pi(D_r(x)),  
\end{equation}
 for every $r>0$ and all $x\in G.$ Then, if $f\in L^{\infty}(\mathbb{R})$ then $f(\pi^{(r)}(\mathcal{R}))=f({r^{\nu}\pi(\mathcal{R})}).$
\end{lemma}
\begin{remark}
Note that if $f_{r}:=r^{-Q}f(r^{-1}\cdot),$ then 
\begin{equation}\label{Eq:dilatedFourier}
    \widehat{f}_{r}(\pi)=\smallint\limits_{G}r^{-Q}f(r^{-1}\cdot x)\pi(x)^*dx=\smallint\limits_{G}f(y)\pi(r\cdot y)^{*}dy=\widehat{f}(r\cdot \pi),
\end{equation}for any $\pi\in \widehat{G}$ and all $r>0,$ with $(r\cdot \pi)(y)=\pi(r\cdot y),$ $y\in G,$ as in \eqref{dilations:repre}.
\end{remark}

The following lemma presents the action of the dilations of the group $G$ into the kernels of bounded functions of a Rockland operator $\mathcal{R},$ see \cite[Page 179]{FischerRuzhanskyBook}.
\begin{lemma}\label{Fundamental:lemmaCZ:graded}
Let $f\in L^{\infty}(\mathbb{R}^{+}_0)$ be a bounded Borel function and let $r>0.$ Then, we have
\begin{equation}
 \forall x\in G,\,   f(r^{\nu}\mathcal{R})\delta(x)=r^{-Q}[f(\mathcal{R})\delta](r^{-1}\cdot x),
\end{equation}where $Q$ is the homogeneous dimension of $G.$
\end{lemma}

\section{Boundedness of the dyadic maximal function} 

In this section we establish the $L^p$-boundedness of the dyadic maximal function associated to a compactly supported Borel measure on a graded Lie group satisfying some additional Fourier transform conditions according to the hypothesis in Theorem \ref{Main:Thn}. First, we start with our main lemma  in the next subsection.
\subsection{The key lemma}
The following Lemma \ref{Lemma:main:dyadic} is our main tool for the proof of Theorem \ref{Main:Thn}.  Indeed, it will be used to use the $L^p$-boundedness of the square function operator in Lemma \ref{Rademacher:lemma} from which we will deduce the boundedness of the dyadic maximal function \eqref{Maximal:Function:Graded}.

\begin{lemma}\label{Lemma:main:dyadic}
Let $K\in L^1(G)$ be a distribution with compact support such that for some $a>0$ the group Fourier transform of $K$ satisfies the growth estimate
\begin{equation}\label{Fourier:order}
  \max_{\pm } \sup_{\pi\in\widehat{G}}\Vert\pi(\mathcal{R})^{\pm \frac{a}{\nu}}\widehat{K}(\pi)\Vert_{\textnormal{op}}<\infty.
\end{equation}
For any $j\in\mathbb{Z},$ let us consider the kernel $K_j(x)=2^{-j Q}K(2^{-j}\cdot x),$ $x\in G,$ and define $ T$ as follows
\begin{equation}
    Tf(x):=\sum_{j=-\infty}^\infty f\ast K_j(x),\,f\in C^\infty_0(G).
\end{equation}Then, $T: L^p(G)\rightarrow L^p(G)$ admits a bounded extension  for all $1<p<\infty.$
\end{lemma}
\begin{proof}
We start the proof by considering a suitable partition of  unity for the spectrum of the Rockland operator $\mathcal{R}.$ For this, let us take a function $\Phi\in \mathscr{S}(\mathbb{R})$ such that (see Lemma 3.13 in \cite{Guorong:Ruz})
\begin{equation}
 \forall\lambda\in (0,\infty),\,  \sum_{j=-\infty}^\infty\Phi(2^{j\nu}\lambda)=1.
\end{equation}Moreover, we can assume that $\Phi$ generates a dyadic partition in the sense that $\textnormal{supp}(\Phi)\subset [1/2^{\nu},2^{\nu}].$ As a consequence we have that
\begin{equation}
    \sum_{j=-\infty}^\infty\Phi(2^{j\nu}\mathcal{R})=I=\textnormal{ identity operator on }L^2(G),
\end{equation}and the convergence in $\mathscr{S}'(G)$ to the Dirac distribution 
\begin{equation}\label{equal:delta}
    \sum_{j=-\infty}^\infty\Phi(2^{j\nu}\mathcal{R})\delta=\delta.
\end{equation}In view of the property in \eqref{Functional:identities:spectral:calculus} for the functional calculus of $\mathcal{R}$, taking the group Fourier transform of \eqref{equal:delta} in both sides  we obtain that
\begin{equation}\label{Identity:decomposition:Hpi}
\forall\pi\in \widehat{G},\,  \sum_{j=-\infty}^\infty \mathscr{F}_{G}[\Phi(\mathcal{R})\delta](2^j\cdot \pi)=   \sum_{j=-\infty}^\infty\Phi[(2^j\cdot \pi)(\mathcal{R})]=I_{H_\pi},
\end{equation}where we have used that, for any $j,$ $(2^j\cdot \pi)(\mathcal{R})=2^{j\nu}\pi(\mathcal{R}).$ To simplify the notation, let us define
\begin{equation}
    \forall x\in G,\forall j\in \mathbb{Z},\,\Phi_j(x):=2^{-jQ}(\Phi(\mathcal{R})\delta)(2^{-j}\cdot),\,\,\Phi(x):=(\Phi(\mathcal{R})\delta)(x).
\end{equation}Then, we have
\begin{equation}\label{Cutt:off}
  \forall\pi\in \widehat{G}\,,\forall j\in \mathbb{Z},\,  \widehat{\Phi}_j(\pi)=\mathscr{F}_{G}[\Phi(\mathcal{R})\delta](2^j\cdot \pi)= \Phi[(2^j\cdot \pi)(\mathcal{R})]= \Phi[2^{j\nu} \pi(\mathcal{R})].
\end{equation}
In view of \eqref{Identity:decomposition:Hpi}, we conclude that
\begin{equation}
   \forall\pi\in \widehat{G},\,   \sum_{j=-\infty}^\infty\widehat{\Phi}_j(\pi)=I_{H_\pi},
\end{equation}and then, in the sense of distributions we have the identity
\begin{equation}\label{Phi:j:parittion}
     \sum_{j=-\infty}^\infty{\Phi}_j=\delta.
\end{equation}
Note also, that if $\{dE_{\pi(\mathcal{R})}\}_{\lambda>0}$ is the spectral measure of the operator
$\pi(\mathcal{R}),$ then
\begin{equation}\label{Functional:calculus}
\forall j\in \mathbb{Z},\, \forall\pi \in \widehat{G},\,\widehat{\Phi}_j(\pi)=   \Phi[(2^j\cdot \pi)(\mathcal{R})]=\Phi[2^{j\nu} \pi(\mathcal{R})]=\smallint\limits_{0}^\infty\Phi(2^{j\nu}\lambda)dE_{\pi(\mathcal{R})}(\lambda).
\end{equation}
These properties of the partition of the unity $\Phi_j,$ $j\in \mathbb{Z},$ will be used in our further analysis. Indeed, we start using \eqref{Phi:j:parittion}
to decompose any $K_j$ as follows:
\begin{equation}\label{Defi:Kj}
    K_j=K_j\ast \delta=  \sum_{k=-\infty}^\infty K_j\ast\Phi_{j+k}.
\end{equation}Define the linear operator $\tilde{T}_{k}$ for any $k\in \mathbb{Z}$ as follows,
\begin{equation}
 \forall f\in C^\infty_0(G),\,   \tilde{T}_{k}f:=\sum_{j=-\infty}^\infty f\ast K_j\ast \Phi_{j+k}.
\end{equation}Note that
\begin{equation*}
    \sum_{k=-\infty}^\infty \tilde{T}_{k}f=\sum_{k=-\infty}^\infty \sum_{j=-\infty}^\infty f\ast K_j\ast \Phi_{j+k}=\sum_{k=-\infty}^\infty \sum_{j=-\infty}^\infty f\ast K_k\ast \Phi_{k+j}=\sum_{k=-\infty}^\infty f\ast K_k=:Tf.
\end{equation*}In view of the last identity we will split our proof in the following steps.
\begin{itemize}
    \item Step 1. To estimate the norm of the operator $\tilde{T}_{k}:L^2(G)\rightarrow L^{2}(G).$ Moreover, we will prove that
    \begin{equation}\label{Step1:proof}
        \exists C>0,\,\forall f\in C^\infty_0(G),\,\Vert \tilde{T}_kf \Vert_{L^2(G)}\leq C 2^{-a|k|} \Vert f\Vert_{L^2(G)}.
    \end{equation}\label{Step2:proof}
    \item Step 2. To estimate the norm of the operator $\tilde{T}_{k}:L^1(G)\rightarrow L^{1,\infty}(G).$ Moreover, we will prove that
    \begin{equation}\label{Weak:1:1:estimate}
        \exists C>0,\,\forall\lambda>0,\,\forall f\in C^\infty_0(G),\, |\{x\in G:|\tilde{T}_k f(x)|>\lambda\}|\leq \frac{C(1+|k|)}{\lambda}\Vert f\Vert_{L^1(G)}.
    \end{equation} 
    \item Step 3. To use Marcinkiewicz interpolation theorem to prove that for any $1<p<2,$
    \begin{equation}\label{Lp:norm:tildeTk}
         \exists C>0,\,\forall f\in C^\infty_0(G),\,\Vert \tilde{T}_kf \Vert_{L^p(G)}\leq C_p 2^{-a|k|\theta}(1+|k|)^{1-\theta} \Vert f\Vert_{L^p(G)},
    \end{equation}where $1/p=\theta/2+(1-\theta).$
    \item Final Step. The proof of Lemma \ref{Lemma:main:dyadic} follows if we sum over $k\in \mathbb{Z}$  both sides of \eqref{Lp:norm:tildeTk}, in the case when
$1<p<2,$ and then  by the duality argument we complete then $L^p$-boundedness of $T$  for $2<p<\infty.$
\end{itemize}Once proved Steps 1 and 2 all the other steps above are clear. It remains, therefore, to prove Steps 1 and 2.

\subsubsection{Step 1.} 
 Let us  prove that
     \begin{equation}
        \exists C>0,\,\forall f\in C^\infty_0(G),\,\Vert \tilde{T}_kf \Vert_{L^2(G)}\leq C 2^{-a|k|} \Vert f\Vert_{L^2(G)}.
    \end{equation} From \eqref{Defi:Kj} we have the identity
    \begin{equation}
        \forall j\in \mathbb{Z},\,\forall\pi\in \widehat{G},\,\widehat{K}_j(\pi)= \sum_{m=-\infty}^{\infty}\widehat{\Phi}_{m+j}(\pi)\widehat{K}_{j}(\pi),
    \end{equation}in the strong topology on $H_\pi.$
Using this fact and Plancherel theorem we deduce that
\begin{equation}
    \Vert \tilde{T}_k f\Vert_{L^2(G)}^2=(\tilde{T}_k f,\tilde{T}_k f)_{L^2(G)}=(\widehat{\tilde{T}_k f},\widehat{\tilde{T}_k f})_{L^2(\widehat{G})},
\end{equation}and by writing the right-hand side of this identity in integral form we have that
\begin{align*}
 \Vert \tilde{T}_k f\Vert_{L^2(G)}^2 &=\sum_{m=-\infty}^{\infty}   \sum_{j=-\infty}^{\infty}  ( \widehat{\Phi}_{m+k}(\pi)\widehat{K}_m(\pi)\widehat{f}(\pi),\widehat{\Phi}_{j+k}(\pi)\widehat{K}_j(\pi)\widehat{f}(\pi))_{L^2(\widehat{G})} \\
 &\leq \sum_{m=-\infty}^{\infty}   \sum_{j=-\infty}^{\infty}  \smallint\limits_{\widehat{G}}|\textnormal{Tr} [\widehat{\Phi}_{m+k}(\pi)\widehat{K}_m(\pi)\widehat{f}(\pi)(\widehat{\Phi}_{j+k}(\pi)\widehat{K}_j(\pi)\widehat{f}(\pi))^*]|d\pi \\
 &=\sum_{j=-\infty}^{\infty}   \sum_{m=-\infty}^{\infty}  \smallint\limits_{\widehat{G}}|\textnormal{Tr} [\widehat{\Phi}_{m+k}(\pi)\widehat{K}_m(\pi)\widehat{f}(\pi)\widehat{f}(\pi)^*\widehat{K}_j(\pi)^*\widehat{\Phi}_{j+k}(\pi)^*]|d\pi. 
\end{align*}Because of \eqref{Cutt:off}, $\widehat{\Phi}_{j+k}(\pi)$ is self-adjoint in every representation space, and the operators $\widehat{\Phi}_{m+k}(\pi), \widehat{\Phi}_{j+k}(\pi),$ commute with each other and with other Borel functions of $\pi(\mathcal{R}),$ and we can write
\begin{align*}
  \Vert \tilde{T}_k f\Vert_{L^2(G)}^2 &\leq    \sum_{j=-\infty}^{\infty}   \sum_{m=-\infty}^{\infty}  \smallint\limits_{\widehat{G}}|\textnormal{Tr} [\widehat{\Phi}_{m+k}(\pi)\widehat{K}_m(\pi)\widehat{f}(\pi)\widehat{f}(\pi)^*\widehat{K}_j(\pi)^*\widehat{\Phi}_{j+k}(\pi)]|d\pi\\
  &=    \sum_{j=-\infty}^{\infty}   \sum_{m=-\infty}^{\infty}  \smallint\limits_{\widehat{G}}|\textnormal{Tr} [\widehat{K}_j(\pi)^*\widehat{\Phi}_{j+k}(\pi)\widehat{\Phi}_{m+k}(\pi)\widehat{K}_m(\pi)\widehat{f}(\pi)\widehat{f}(\pi)^*]|d\pi\\
   &=     \sum_{j=-\infty}^{\infty}   \sum_{m=-\infty}^{\infty}  \smallint\limits_{\widehat{G}}|\textnormal{Tr} [\widehat{K}(2^j\cdot \pi)^*[(2^j\cdot \pi)(\mathcal{R})^{\pm \frac{a}{\nu}}][(2^j\cdot \pi)(\mathcal{R})^{\mp \frac{a}{\nu}}]\\
   &\hspace{4cm}\widehat{\Phi}_{j+k}(\pi)\widehat{\Phi}_{m+k}(\pi)\widehat{K}_m(\pi)\widehat{f}(\pi)\widehat{f}(\pi)^*]|d\pi\\
   &\leq      \sum_{j=-\infty}^{\infty}   \sum_{m=-\infty}^{\infty}  \smallint\limits_{\widehat{G}} \| \widehat{K}(2^j\cdot \pi)^*[(2^j\cdot \pi)(\mathcal{R})^{\pm \frac{a}{\nu}}] \|_{\textnormal{op}}  |\textnormal{Tr} [ [(2^j\cdot \pi)(\mathcal{R})^{\mp \frac{a}{\nu}}]\\
   &\hspace{4cm}\widehat{\Phi}_{j+k}(\pi)\widehat{\Phi}_{m+k}(\pi)\widehat{K}_m(\pi)\widehat{f}(\pi)\widehat{f}(\pi)^*]|d\pi\\
   &=      \sum_{j=-\infty}^{\infty}   \sum_{m=-\infty}^{\infty}  \smallint\limits_{\widehat{G}} \| [[(2^j\cdot \pi)(\mathcal{R})^{\pm \frac{a}{\nu}}]\widehat{K}(2^j\cdot \pi)]^* \|_{\textnormal{op}}  |\textnormal{Tr} [ [(2^j\cdot \pi)(\mathcal{R})^{\mp \frac{a}{\nu}}]\\
   &\hspace{4cm}\widehat{\Phi}_{j+k}(\pi)\widehat{\Phi}_{m+k}(\pi)\widehat{K}_m(\pi)\widehat{f}(\pi)\widehat{f}(\pi)^*]|d\pi\\
   &=      \sum_{j=-\infty}^{\infty}   \sum_{m=-\infty}^{\infty}  \smallint\limits_{\widehat{G}} \| [(2^j\cdot \pi)(\mathcal{R})^{\pm \frac{a}{\nu}}]\widehat{K}(2^j\cdot \pi) \|_{\textnormal{op}}  |\textnormal{Tr} [ [(2^j\cdot \pi)(\mathcal{R})^{\mp \frac{a}{\nu}}]\\
   &\hspace{4cm}\widehat{\Phi}_{j+k}(\pi)\widehat{\Phi}_{m+k}(\pi)\widehat{K}_m(\pi)\widehat{f}(\pi)\widehat{f}(\pi)^*]|d\pi\\
   &\lesssim     \sum_{j=-\infty}^{\infty}   \sum_{m=-\infty}^{\infty}  \smallint\limits_{\widehat{G}}  |\textnormal{Tr} [ [(2^j\cdot \pi)(\mathcal{R})^{\mp \frac{a}{\nu}}]\\
   &\hspace{4cm}\widehat{\Phi}_{j+k}(\pi)\widehat{\Phi}_{m+k}(\pi)\widehat{K}_m(\pi)\widehat{f}(\pi)\widehat{f}(\pi)^*]|d\pi\\
   &=     \sum_{j=-\infty}^{\infty}   \sum_{m=-\infty}^{\infty}  \smallint\limits_{\widehat{G}}  |\textnormal{Tr} [ \widehat{\Phi}_{j+k}(\pi)\widehat{\Phi}_{m+k}(\pi)[(2^j\cdot \pi)(\mathcal{R})^{\mp \frac{a}{\nu}}]\widehat{K}_m(\pi)\widehat{f}(\pi)\widehat{f}(\pi)^*]|d\pi\\
   &=     \sum_{j=-\infty}^{\infty}   \sum_{m=-\infty}^{\infty}  \smallint\limits_{\widehat{G}}  |\textnormal{Tr} [ [\widehat{f}(\pi)\widehat{f}(\pi)^*]\widehat{\Phi}_{j+k}(\pi)\widehat{\Phi}_{m+k}(\pi)[(2^j\cdot \pi)(\mathcal{R})^{\mp \frac{a}{\nu}}]\widehat{K}_m(\pi)]|d\pi,
\end{align*}
where we have used the kernel condition
\begin{equation}\label{First:auxiliar:inequality}
    \sup_{j\in \mathbb{Z}; \,\pi\in \widehat{G}}\| [(2^j\cdot \pi)(\mathcal{R})^{\mp\frac{a}{\nu}}]\widehat{K}(2^j\cdot \pi) \|_{\textnormal{op}}  \leq \sup_{\pi\in \widehat{G}} \|  \pi(\mathcal{R})^{\mp \frac{a}{\nu}}\widehat{K}( \pi) \|_{\textnormal{op}}<\infty.  
\end{equation}
In view of \eqref{Functional:calculus}, the properties of the functional calculus allow us to write
\begin{equation}\label{product:finite}
    \widehat{\Phi}_{j+k}(\pi)\widehat{\Phi}_{m+k}(\pi)=  \smallint\limits_{0}^\infty\Phi(2^{(j+k)\nu}\lambda)\Phi(2^{(m+k)\nu}\lambda)dE_{\pi(\mathcal{R})}(\lambda).
\end{equation}
Since the support of $\Phi$ lies in the interval $[1/2^{\nu},2^{\nu}]$ we have that 
\begin{equation*}
    \forall\ell,\,\ell',\textnormal{  such that  }|\ell-\ell'|\geq 2,\,\, \Phi_{\ell}(\lambda)\Phi_{\ell'}(\lambda)=0.
\end{equation*}In consequence the representation in \eqref{product:finite} shows that $ \widehat{\Phi}_{j+k}(\pi)\widehat{\Phi}_{m+k}(\pi)\equiv 0_{H_\pi}$ if $|m-j|\geq 2.$ So, we have that 
\begin{equation}\label{The:integral:L2:tobe:estimated}
    \Vert \tilde{T}_k f\Vert_{L^2(G)}^2\lesssim \sum_{j=-\infty}^\infty\sum_{m=j-1}^{j+1}  \smallint\limits_{\widehat{G}}  |\textnormal{Tr} [ [\widehat{f}(\pi)\widehat{f}(\pi)^*]\widehat{\Phi}_{j+k}(\pi)\widehat{\Phi}_{m+k}(\pi)[(2^j\cdot \pi)(\mathcal{R})^{\mp \frac{a}{\nu}}]\widehat{K}_m(\pi)]|d\pi.
\end{equation} 
In view of the identity  $(2^j\cdot \pi)(\mathcal{R})=2^{j\nu}\times \pi(\mathcal{R}),$ $j\in \mathbb{Z},$  we have that 
\begin{align*}
 & \mathscr{A}(\pi):=  |\textnormal{Tr} [[\widehat{f}(\pi)\widehat{f}(\pi)^*]  \widehat{\Phi}_{j+k}(\pi)\widehat{\Phi}_{m+k}(\pi)[(2^j\cdot \pi)(\mathcal{R})^{\mp \frac{a}{\nu}}]\widehat{K}_m(\pi)]| \\
 &=|\textnormal{Tr} [ [\widehat{f}(\pi)\widehat{f}(\pi)^*]\widehat{\Phi}_{j+k}(\pi)\widehat{\Phi}_{m+k} [(2^{j}\cdot\pi)(\mathcal{R})]^{\mp \frac{a}{\nu}}(\pi)\widehat{K}(2^m\cdot\pi) ]|\\
   &=  |\textnormal{Tr} [ [\widehat{f}(\pi)\widehat{f}(\pi)^*]\widehat{\Phi}_{j+k}(\pi)\widehat{\Phi}_{m+k}(\pi) [2^{j\nu}\times \pi(\mathcal{R})]^{\mp \frac{a}{\nu}}\widehat{K}(2^m\cdot\pi)]|\\
    &=  |\textnormal{Tr} [ [\widehat{f}(\pi)\widehat{f}(\pi)^*] \widehat{\Phi}_{j+k}(\pi)\widehat{\Phi}_{m+k}(\pi) [2^{(j-m)\nu}\times 2^{m\nu}\times  \pi(\mathcal{R})]^{\mp \frac{a}{\nu}} \widehat{K}(2^m\cdot\pi)]|\\
    &= 2^{\mp (j-m)a}  |\textnormal{Tr} [ [\widehat{f}(\pi)\widehat{f}(\pi)^*] \widehat{\Phi}_{j+k}(\pi)\widehat{\Phi}_{m+k}(\pi) [ 2^{m\nu}\times  \pi(\mathcal{R})]^{\mp \frac{a}{\nu}}\widehat{K}(2^m\cdot\pi)]|.
\end{align*}Since $j-m=0,\pm 1,$ $2^{\mp 2(j-m)a}\asymp 1,$ we have
\begin{align*}
     & \mathscr{A}(\pi):=  |\textnormal{Tr} [[\widehat{f}(\pi)\widehat{f}(\pi)^*]  \widehat{\Phi}_{j+k}(\pi)\widehat{\Phi}_{m+k}(\pi)[(2^j\cdot \pi)(\mathcal{R})^{\mp \frac{a}{\nu}}]\widehat{K}_m(\pi)]|\\
     &=2^{\mp (j-m)a}  |\textnormal{Tr} [ [\widehat{f}(\pi)\widehat{f}(\pi)^*] \widehat{\Phi}_{j+k}(\pi)\widehat{\Phi}_{m+k}(\pi) [ 2^{m\nu}\times  \pi(\mathcal{R})]^{\mp \frac{a}{\nu}}\widehat{K}(2^m\cdot\pi)]|\\
     &\asymp |\textnormal{Tr} [ [\widehat{f}(\pi)\widehat{f}(\pi)^*] \widehat{\Phi}_{j+k}(\pi)\widehat{\Phi}_{m+k}(\pi) [ 2^{m\nu}\times  \pi(\mathcal{R})]^{\mp \frac{a}{\nu}}\widehat{K}(2^m\cdot\pi)]|.
\end{align*} Until now we have proved that the term
$$   \mathscr{A}(\pi):= |\textnormal{Tr} [[\widehat{f}(\pi)\widehat{f}(\pi)^*]  \widehat{\Phi}_{j+k}(\pi)\widehat{\Phi}_{m+k}(\pi)[(2^j\cdot \pi)(\mathcal{R})^{\mp \frac{a}{\nu}}]\widehat{K}_m(\pi)]|$$
can be estimated as follows
\begin{equation}\label{First:estimate}
   \mathscr{A}(\pi)\asymp |\textnormal{Tr} [ [\widehat{f}(\pi)\widehat{f}(\pi)^*] \widehat{\Phi}_{j+k}(\pi)\widehat{\Phi}_{m+k}(\pi) [ 2^{m\nu}\times  \pi(\mathcal{R})]^{\mp \frac{a}{\nu}}\widehat{K}(2^m\cdot\pi)]|.
\end{equation}
To estimate the right-hand side of \eqref{First:estimate} let us use \eqref{First:auxiliar:inequality}
 again as follows
\begin{align*}
  \mathscr{A}(\pi) &  \asymp |\textnormal{Tr} [ [\widehat{f}(\pi)\widehat{f}(\pi)^*] \widehat{\Phi}_{j+k}(\pi)\widehat{\Phi}_{m+k}(\pi)[ (2^{m}\cdot  \pi)(\mathcal{R})]^{\mp \frac{2a}{\nu}}[ (2^{m}\cdot  \pi)(\mathcal{R})]^{\pm \frac{a}{\nu}}\widehat{K}(2^m\cdot\pi) ]|\\
  &\leq  \textnormal{Tr} [ |[\widehat{f}(\pi)\widehat{f}(\pi)^*] \widehat{\Phi}_{j+k}(\pi)\widehat{\Phi}_{m+k}(\pi)[ (2^{m}\cdot  \pi)(\mathcal{R})]^{\mp \frac{2a}{\nu}}|]\Vert [ (2^{m}\cdot  \pi)(\mathcal{R})]^{\pm \frac{a}{\nu}}\widehat{K}(2^m\cdot\pi) \Vert_{\textnormal{op}}\\
  &\lesssim  \textnormal{Tr} [ |[\widehat{f}(\pi)\widehat{f}(\pi)^*] \widehat{\Phi}_{j+k}(\pi)\widehat{\Phi}_{m+k}(\pi)[ (2^{m}\cdot  \pi)(\mathcal{R})]^{\mp \frac{2a}{\nu}}|].
\end{align*}
Now, we have the better estimate (because it only depends on the functional calculus of the symbol $\pi(\mathcal{R})$ of the Rockland operator $\mathcal{R}$ and of $\widehat{f}$)
\begin{equation}\label{partial:II}
 \mathscr{A}(\pi)\lesssim \textnormal{Tr} [ |[\widehat{f}(\pi)\widehat{f}(\pi)^*] \widehat{\Phi}_{j+k}(\pi)\widehat{\Phi}_{m+k}(\pi)[ (2^{m}\cdot  \pi)(\mathcal{R})]^{\mp \frac{2a}{\nu}}|].   
\end{equation}
Note that when composing $\widehat{\Phi}_{j+k}(\pi)\widehat{\Phi}_{m+k}(\pi) $ with the operator $[ (2^{m}\cdot  \pi)(\mathcal{R})]^{\pm \frac{2a}{\nu}}$ we remove the zero from the spectrum of the new operator
\begin{equation}\label{operator:1:aux}
\widehat{\Phi}_{j+k}(\pi)\widehat{\Phi}_{m+k}(\pi)[ (2^{m}\cdot  \pi)(\mathcal{R})]^{\pm \frac{2a}{\nu}}    
\end{equation}
 because of the spectral identity in \eqref{Cutt:off} and the properties of the support of $\Phi.$ 
From \eqref{partial:II} we also have the estimate
\begin{equation}\label{partial:III}
 \mathscr{A}(\pi)\lesssim \textnormal{Tr} [ |[\widehat{f}(\pi)\widehat{f}(\pi)^*] \widehat{\Phi}_{j+k}(\pi)\widehat{\Phi}_{m+k}(\pi)[ (2^{m}\cdot  \pi)(\mathcal{R})]^{- \frac{2a}{\nu}}|].   
\end{equation}
Now, using the functional calculus for Rockland operators we can estimate \eqref{partial:III}. Indeed, we have
\begin{align*}
    &\textnormal{Tr} [ |[\widehat{f}(\pi)\widehat{f}(\pi)^*] \widehat{\Phi}_{j+k}(\pi)\widehat{\Phi}_{m+k}(\pi)[(2^{m}\cdot\pi)(\mathcal{R})]^{-\frac{2a}{\nu}}|]\\
    & =\textnormal{Tr} [ |[\widehat{f}(\pi)\widehat{f}(\pi)^*] \widehat{\Phi}_{j+k}(\pi)\widehat{\Phi}_{m+k}(\pi)[(2^{m\nu}\pi)(\mathcal{R})]^{-\frac{2a}{\nu}}|]\\
    &=\textnormal{Tr} [ |[\widehat{f}(\pi)\widehat{f}(\pi)^*] \widehat{\Phi}_{j+k}(\pi)\widehat{\Phi}_{m+k}(\pi)2^{-2ma}[\pi(\mathcal{R})]^{-\frac{2a}{\nu}}|]\\
    &=2^{-2ma}\textnormal{Tr} [ |[\widehat{f}(\pi)\widehat{f}(\pi)^*]\widehat{\Phi}_{j+k}(\pi)\widehat{\Phi}_{m+k}(\pi)[\pi(\mathcal{R})]^{-\frac{2a}{\nu}}|]\\
    &\asymp 2^{-2ja}\textnormal{Tr} [ |[\widehat{f}(\pi)\widehat{f}(\pi)^*] \widehat{\Phi}_{j+k}(\pi)\widehat{\Phi}_{m+k}(\pi)[\pi(\mathcal{R})]^{-\frac{2a}{\nu}}|].
\end{align*}  

Summarising we have proved that
\begin{equation}\label{Fisrt:implication:auxiliar}
    \mathscr{A}(\pi) \lesssim 2^{-2ja}\textnormal{Tr} [ |[\widehat{f}(\pi)\widehat{f}(\pi)^*] \widehat{\Phi}_{j+k}(\pi)\widehat{\Phi}_{m+k}(\pi)[\pi(\mathcal{R})]^{-\frac{2a}{\nu}}|].
\end{equation}
Note that in the same way that \eqref{partial:II} implies \eqref{Fisrt:implication:auxiliar} by changing $-a$ by $+a$ in the argument of this implication we also have that \eqref{partial:II} implies 
\begin{equation}
\label{second:implication:auxiliar}
     \mathscr{A}(\pi) \lesssim 2^{2ja}\textnormal{Tr} [ |[\widehat{f}(\pi)\widehat{f}(\pi)^*] \widehat{\Phi}_{j+k}(\pi)\widehat{\Phi}_{m+k}(\pi)[\pi(\mathcal{R})]^{\frac{2a}{\nu}}|].
\end{equation}
Then, from \eqref{Fisrt:implication:auxiliar} and \eqref{second:implication:auxiliar} we have the following similar bounds
\begin{equation}\label{similar:I}
   \mathscr{A}(\pi)\lesssim 2^{-2ja}\max_{\pm }\textnormal{Tr} [ |[\widehat{f}(\pi)\widehat{f}(\pi)^*] \widehat{\Phi}_{j+k}(\pi)\widehat{\Phi}_{m+k}(\pi)[\pi(\mathcal{R})]^{\pm \frac{2a}{\nu}}|], 
\end{equation}
    and 
\begin{equation}\label{similar:II}
   \mathscr{A}(\pi)\lesssim2^{2ja}\max_{\pm }\textnormal{Tr} [ |[\widehat{f}(\pi)\widehat{f}(\pi)^*] \widehat{\Phi}_{j+k}(\pi)\widehat{\Phi}_{m+k}(\pi)[\pi(\mathcal{R})]^{\pm \frac{2a}{\nu}}|]. 
\end{equation}Note that \eqref{similar:I} and \eqref{similar:II} 
imply the estimate
\begin{align*}&\mathscr{A}(\pi)\lesssim
\min\{2^{2ja},2^{-2ja}\}\max_{\pm }\textnormal{Tr} [ |[\widehat{f}(\pi)\widehat{f}(\pi)^*] \widehat{\Phi}_{j+k}(\pi)\widehat{\Phi}_{m+k}(\pi)[\pi(\mathcal{R})]^{\pm \frac{2a}{\nu}}|].   
\end{align*}
Since $\min\{2^{2ja},2^{-2ja}\}=2^{-2|j|a},$ we have deduced the estimate
 \begin{equation}\label{similar:III}
   \mathscr{A}(\pi)\lesssim2^{-2|j|a}\max_{\pm }\textnormal{Tr} [ |[\widehat{f}(\pi)\widehat{f}(\pi)^*] \widehat{\Phi}_{j+k}(\pi)\widehat{\Phi}_{m+k}(\pi)[\pi(\mathcal{R})]^{\pm \frac{2a}{\nu}}|]. 
\end{equation} Now in order to estimate   \eqref{similar:III} we can use the functional calculus for Rockland operators. 
Indeed, since $m\asymp j$  note that
\begin{align*}
    &\textnormal{Tr} [ |[\widehat{f}(\pi)\widehat{f}(\pi)^*]  \widehat{\Phi}_{j+k}(\pi)\widehat{\Phi}_{m+k}(\pi)\pi(\mathcal{R})^{\frac{\pm 2 a}{\nu}}|]\\
    &=\textnormal{Tr} [ |[\widehat{f}(\pi)\widehat{f}(\pi)^*]  \left( \smallint\limits_{0}^\infty\Phi(2^{(j+k)\nu}\lambda)\Phi(2^{(m+k)\nu}\lambda)dE_{\pi(\mathcal{R})}(\lambda)\right)  \pi(\mathcal{R})^{\frac{\pm 2a}{\nu}}|]\\
    &=\textnormal{Tr} [ |[\widehat{f}(\pi)\widehat{f}(\pi)^*]  \left( \smallint\limits_{0}^\infty\Phi(2^{(j+k)\nu}\lambda)\Phi(2^{(m+k)\nu}\lambda)\lambda^{\frac{\pm 2a}{\nu}}dE_{\pi(\mathcal{R})}(\lambda)\right)  |]\\
    & \asymp \textnormal{Tr} [ |[\widehat{f}(\pi)\widehat{f}(\pi)^*]  \left(\,\,\smallint\limits_{\lambda\sim 2^{-(j+k)\nu}}\Phi(2^{(j+k)\nu}\lambda)\Phi(2^{(m+k)\nu}\lambda)\lambda^{\frac{\pm 2a}{\nu}}dE_{\pi(\mathcal{R})}(\lambda)\right)  |],
\end{align*}where we have used the notation $\lambda\sim 2^{-(j+k)\nu} $ to indicate that $\lambda\in [2^{-(j+k+1)\nu},2^{-(j+k-1)\nu}].$ Using the properties of the spectral projections $E_{\pi(\mathcal{R})}$ we have that
\begin{align*}  
&\textnormal{Tr} [ |[\widehat{f}(\pi)\widehat{f}(\pi)^*]  \left(\,\,\smallint\limits_{\lambda\sim 2^{-(j+k)\nu}}\Phi(2^{(j+k)\nu}\lambda)\Phi(2^{(m+k)\nu}\lambda)dE_{\pi(\mathcal{R})}(\lambda)\right)  |]\\
&=\textnormal{Tr} [ |[\widehat{f}(\pi)\widehat{f}(\pi)^*] E_{\pi(\mathcal{R})}[2^{-(j+k+1)\nu},2^{-(j+k-1)\nu}]  \left(\,\,\smallint\limits_{\lambda\sim 2^{-(j+k)\nu}}\Phi(2^{(j+k)\nu}\lambda)\Phi(2^{(m+k)\nu}\lambda)dE_{\pi(\mathcal{R})}(\lambda)\right)  |]\\
&\leq \textnormal{Tr} [ |[\widehat{f}(\pi)\widehat{f}(\pi)^*] E_{\pi(\mathcal{R})}[2^{-(j+k+1)\nu},2^{-(j+k-1)\nu}]|]\\
&\hspace{2cm}\times \Vert\left(\,\,\smallint\limits_{\lambda\sim 2^{-(j+k)\nu}}\Phi(2^{(j+k)\nu}\lambda)\Phi(2^{(m+k)\nu}\lambda)dE_{\pi(\mathcal{R})}(\lambda)\right)  \|_{\textnormal{op}}.
\end{align*}
Using that $\lambda^{\frac{\pm 2 a}{\nu}}\sim 2^{\mp 2(j+k) a }$ when $\lambda\sim 2^{-(j+k)\nu},$ we have that 
\begin{align*}
  &\Vert \left( \,\,\smallint\limits_{\lambda\sim 2^{-(j+k)\nu}}\Phi(2^{(j+k)\nu}\lambda)\Phi(2^{(m+k)\nu}\lambda)\lambda^{\frac{\pm 2 a}{\nu}}dE_{\pi(\mathcal{R})}(\lambda)\right)  \|_{\textnormal{op}}\\
  \\
  &\asymp  2^{\mp 2(j+k) a }\Vert \left(\,\, \smallint\limits_{\lambda\sim 2^{-(j+k)\nu}}\Phi(2^{(j+k)\nu}\lambda)\Phi(2^{(m+k)\nu}\lambda)dE_{\pi(\mathcal{R})}(\lambda)\right)  \|_{\textnormal{op}}\\
  &\leq  2^{\mp 2(j+k) a }\Vert \left(\,\, \smallint\limits_{\lambda\geq 0}\Phi(2^{(j+k)\nu}\lambda)\Phi(2^{(m+k)\nu}\lambda)dE_{\pi(\mathcal{R})}(\lambda)\right)  \|_{\textnormal{op}}\\
  &= 2^{\mp 2(j+k) a }\Vert \Phi(2^{(j+k)\nu}\pi(\mathcal{R}))\Phi(2^{(m+k)\nu}\pi(\mathcal{R})) \|_{\textnormal{op}}\\
  &\leq 2^{\mp 2(j+k) a } \sup_{\lambda\geq0 }(\Phi(2^{(j+k)\nu}(\lambda))\Phi(2^{(m+k)\nu}\lambda)\\
  &\leq 2^{\mp 2(j+k) a } \Vert\Phi\Vert_{L^\infty}^2.
\end{align*}
So, we have proved the inequality
\begin{align*}
 &\textnormal{Tr} [ |[\widehat{f}(\pi)\widehat{f}(\pi)^*]  \left(\,\,\smallint\limits_{\lambda\sim 2^{-(j+k)\nu}}\Phi(2^{(j+k)\nu}\lambda)\Phi(2^{(m+k)\nu}\lambda)dE_{\pi(\mathcal{R})}(\lambda)\right)  |]   \\
 &\leq 2^{\mp 2(j+k) a } \Vert\Phi\Vert_{L^\infty}^2 \textnormal{Tr} [ |[\widehat{f}(\pi)\widehat{f}(\pi)^*] E_{\pi(\mathcal{R})}[2^{-(j+k+1)\nu},2^{-(j+k-1)\nu}]|],
\end{align*}
from where we deduce the estimate
\begin{align*}
   & \max_{\pm }\textnormal{Tr} [ |[\widehat{f}(\pi)\widehat{f}(\pi)^*]  \left(\,\,\smallint\limits_{\lambda\sim 2^{-(j+k)\nu}}\Phi(2^{(j+k)\nu}\lambda)\Phi(2^{(m+k)\nu}\lambda)dE_{\pi(\mathcal{R})}(\lambda)\right)  |] \\
    &\lesssim_{\Phi} 2^{ -2|j+k| a }\textnormal{Tr} [ |[\widehat{f}(\pi)\widehat{f}(\pi)^*] E_{\pi(\mathcal{R})}[2^{-(j+k+1)\nu},2^{-(j+k-1)\nu}]|].
    \end{align*}
The previous analysis allows us to estimate \eqref{similar:III}  as follows 
\begin{align*}
     \mathscr{A}(\pi)&\lesssim2^{-2|j|a}\max_{\pm }\textnormal{Tr} [ |[\widehat{f}(\pi)\widehat{f}(\pi)^*]  \left(\,\,\smallint\limits_{\lambda\sim 2^{-(j+k)\nu}}\Phi(2^{(j+k)\nu}\lambda)\Phi(2^{(m+k)\nu}\lambda)dE_{\pi(\mathcal{R})}(\lambda)\right)  |]\\
     &\lesssim_{\Phi,K,a}2^{-2|j|a}\times 2^{-2|j+k|a} \textnormal{Tr} [ |[\widehat{f}(\pi)\widehat{f}(\pi)^*] E_{\pi(\mathcal{R})}[2^{-(j+k+1)\nu},2^{-(j+k-1)\nu}]|].
\end{align*}Using the triangle inequality $|j+k|\geq |k|-|j|$ we have the reverse inequality
$$  -|j+k|\leq |j|-|k|,$$ and then
\begin{align*}    
\mathscr{A}(\pi)&\lesssim_{\Phi,K,a}2^{-2|j|a}\times 2^{-2|k|a+2|j|a}\textnormal{Tr} [ |[\widehat{f}(\pi)\widehat{f}(\pi)^*] E_{\pi(\mathcal{R})}[2^{-(j+k+1)\nu},2^{-(j+k-1)\nu}]|]\\
&=2^{-2|k|a}\textnormal{Tr} [ |[\widehat{f}(\pi)\widehat{f}(\pi)^*] E_{\pi(\mathcal{R})}[2^{-(j+k+1)\nu},2^{-(j+k-1)\nu}]|].
\end{align*}
In consequence, coming back to \eqref{The:integral:L2:tobe:estimated} we have that
\begin{align*}
    \Vert \tilde{T}_k f\Vert_{L^2(G)}^2 &\leq \sum_{j=-\infty}^\infty\sum_{m=j-1}^{j+1}\smallint\limits_{\widehat{G}}  \|  \widehat{\Phi}_{j+k}(\pi)\widehat{\Phi}_{m+k}(\pi)[(2^j\cdot \pi)(\mathcal{R})^{\pm\frac{a}{\nu}}]\widehat{K}_m(\pi)\|_{\textnormal{op} }     \|\widehat{f}(\pi)|_{ \textnormal{HS} }^2\,d\pi\\
    &\lesssim\sum_{j=-\infty}^\infty\sum_{m=j-1}^{j+1} 2^{-2|k|a}\smallint\limits_{\widehat{G}}\textnormal{Tr} [ |[\widehat{f}(\pi)\widehat{f}(\pi)^*] E_{\pi(\mathcal{R})}[2^{-(j+k+1)\nu},2^{-(j+k-1)\nu}]|]d\pi\\
    &\lesssim\sum_{j=-\infty}^\infty 2^{-2|k|a}\smallint\limits_{\widehat{G}}\textnormal{Tr} [ |[\widehat{f}(\pi)\widehat{f}(\pi)^*] E_{\pi(\mathcal{R})}[2^{-(j+k+1)\nu},2^{-(j+k-1)\nu}]|]d\pi.
\end{align*} 
Using the fact that the mapping $t_{j}:=j+\cdot:\mathbb{Z}\rightarrow \mathbb{Z},\,k\mapsto j+k, $ is a bijection on $\mathbb{Z}$, we have the estimates
\begin{align*}
&\Vert \tilde{T}_k f\Vert_{L^2(G)}^2\lesssim \sum_{j=-\infty}^\infty 2^{-2|k|a}\smallint\limits_{\widehat{G}}\textnormal{Tr} [ |[\widehat{f}(\pi)\widehat{f}(\pi)^*] E_{\pi(\mathcal{R})}[2^{-(j+1)\nu},2^{-(j-1)\nu}]|]d\pi\\
&\asymp  2^{-2|k|a}\smallint\limits_{\widehat{G}}\textnormal{Tr} [ |[\widehat{f}(\pi)\widehat{f}(\pi)^*] E_{\pi(\mathcal{R})}(-\infty,\infty)|]d\pi\\
&=  2^{-2|k|a}\smallint\limits_{\widehat{G}}\textnormal{Tr} [ |[\widehat{f}(\pi)\widehat{f}(\pi)^*] I_{H_\pi}]d\pi\\
&=2^{-2|k|a}\Vert \widehat{f}\Vert_{L^2(\widehat{G})}^2.
\end{align*}Then, using the last inequality and the Plancherel theorem we deduce that the $L^2$-norm of $\tilde{T}_k f$ satisfies the estimate
\begin{equation*}
    \Vert \tilde{T}_k f\Vert_{L^2(G)}^2 \lesssim_{\Phi,K,a} 2^{-2|k|a}\Vert f\Vert^2_{L^2(G)},
\end{equation*}which implies 
\begin{equation}
    \Vert \tilde{T}_k f\Vert_{L^2(G)} \lesssim_{\Phi,K,a} 2^{-|k|a}\Vert f\Vert_{L^2(G)}.
\end{equation} The proof of \eqref{Step1:proof} is complete and we have concluded Step 1 of the proof.
\subsubsection{ Step 2.} For the proof of the {\it weak (1,1) estimate} we will use the fundamental theorem for singular integrals due to Coifman and Weiss \cite{CoifmanWeiss1971}. We write
\begin{equation*}
 \forall f\in C^\infty_0(G),\,   \tilde{T}_{k}f:=f\ast (\tilde{T}_{k}\delta),\, (\tilde{T}_{k}\delta):=\sum_{j=-\infty}^\infty K_j\ast \Phi_{j+k}.
\end{equation*} We shall prove the estimate
\begin{equation}\label{auxiliar:weak:1:1}
  [\tilde{T}_{k}\delta]_{H}:= \sup_{|y|\leq 1} \smallint\limits_{|x|>2|y|}|\tilde{T}_{k}\delta(y^{-1}x)-\tilde{T}_{k}\delta(x)|dx\lesssim_{K}(1+|k|),
\end{equation}and then the constant $(1+|k|)$ of the right-hand side of \eqref{auxiliar:weak:1:1} gives the estimate in \eqref{Weak:1:1:estimate}.  
Indeed, note that the estimate in \eqref{auxiliar:weak:1:1} says that the kernel $\Tilde{T}_k\delta$ of $\Tilde{T}_k$ satisfies the H\"ormander condition $ [\tilde{T}_{k}\delta]_{H}\lesssim_{K}(1+|k|).$ Then, from the fundamental theorem for singular integrals due to Coifman and Weiss \cite{CoifmanWeiss1971}, one has that the $(L^1,L^{1,\infty})$-operator norm of $T$ satisfies the estimate
\begin{equation}
    \Vert \Tilde{T}_k\Vert_{L^1\rightarrow L^{1,\infty}}\lesssim \Vert \Tilde{T}_k\Vert_{L^2\rightarrow L^{2}}+ [\tilde{T}_{k}\delta]_{H}\lesssim 2^{-|k|a}+(1+|k|)\lesssim (1+|k|),
\end{equation}which proves \eqref{Weak:1:1:estimate}.

Note that for any $y\in G$ with $|y|\leq 1,$ we have
\begin{align*}
    \smallint\limits_{|x|>2|y|}|\tilde{T}_{k}\delta(y^{-1}x)-\tilde{T}_{k}\delta(x)|dx &=\smallint\limits_{|x|>2|y|}\left|\sum_{j=-\infty}^\infty (K_j\ast \Phi_{j+k})(y^{-1}x)-(K_j\ast \Phi_{j+k})(x)\right|dx\\
    &\leq \sum_{j=-\infty}^\infty \smallint\limits_{|x|>2|y|} \left|(K_j\ast \Phi_{j+k})(y^{-1}x)-(K_j\ast \Phi_{j+k})(x)\right|dx\\
     &= \sum_{j=-\infty}^\infty I_{j,k},
\end{align*}where 
$$ I_{j,k}=   \smallint\limits_{|x|>2|y|} \left|(K_j\ast \Phi_{j+k})(y^{-1}x)-(K_j\ast \Phi_{j+k})(x)\right|dx.$$
Note that, by the Hausdorff-Young inequality, we have the following immediate estimate
\begin{align*}
   I_{j,k} &\leq  \smallint\limits_{G} \left|(K_j\ast \Phi_{j+k})(y^{-1}x)-(K_j\ast \Phi_{j+k})(x)\right|dx \leq 2\smallint\limits_{G} \left|(K_j\ast \Phi_{j+k})(z)\right|dz\\
   &\leq \Vert K_j\Vert_{L^1(G)} \Vert \Phi_{j+k}\Vert_{L^1(G)}= \Vert K\Vert_{L^1(G)} \Vert \Phi(\mathcal{R})\delta\Vert_{L^1(G)}.
\end{align*}
Indeed, since $$ K_j:=2^{-jQ}K(2^{-j}\cdot) \textnormal{ and  } \Phi_{j+k}:=2^{-(j+k)Q}(\Phi(\mathcal{R})\delta)(2^{-(j+k)}\cdot),$$ the changes of variables $z=2^{-j}\cdot x$ and $z'=2^{-(j+k)}\cdot x'$ imply the equality
\begin{equation*}
 \Vert K_j\Vert_{L^1(G)}=  \smallint\limits_{G}|K_j(x)|dx= 2^{-jQ}\smallint\limits_{G}|K(2^{-j}\cdot x)|dx=\smallint\limits_{G}|K(z)|dz=\Vert K\Vert_{L^1(G)},
\end{equation*}as well as that
$$
 \displaystyle \Vert \Phi_{j+k}\Vert_{L^1(G)}= 2^{-(j+k)Q}\smallint\limits_{G}|(\Phi(\mathcal{R})\delta)(2^{-(j+k)}\cdot x')|dx'=\smallint\limits_{G}|(\Phi(\mathcal{R})\delta)(z')|dz'
$$
$$=\Vert (\Phi(\mathcal{R})\delta)\Vert_{L^1(G)}.$$
Then, we have proved that
\begin{align}\label{Trivial:estimate:for:}
 \forall 0<|y|\leq 1,\,\forall (j,k),\,\,\,   I_{j,k} \lesssim_{\Phi}  \Vert K\Vert_{L^1(G)}
 \end{align}
\noindent Also, we will provide other estimates for $I_{j,k} $ as follows. First, we use the dilation property
\begin{equation}\label{dilation:property}
     \forall r>0,\forall x,y\in G,\,\,r\cdot (xy)=(r\cdot x)(r\cdot y),\textnormal{ and }\,\,\,r\cdot x^{-1}=(r\cdot x)^{-1}.
\end{equation}
Then 
\begin{align*}
    &I_{j,k} =   \smallint\limits_{|x|>2|y|} \left|(K_j\ast \Phi_{j+k})(y^{-1}x)-(K_j\ast \Phi_{j+k})(x)\right|dx\\
     &=   \smallint\limits_{|x|>2|y|} | \smallint\limits_{G}K_j(y^{-1}xz^{-1}) \Phi_{j+k}(z)dz-\smallint\limits_{G}K_j(xz^{-1})\Phi_{j+k}(z)dz|dx\\
      &=   \smallint\limits_{|x|>2|y|}2^{-jQ} | \smallint\limits_{G}K(2^{-j}\cdot (y^{-1}xz^{-1})  ) 2^{-(j+k)Q})(\Phi(\mathcal{R})\delta)(2^{-j-k}z)dz\\
      &\hspace{4cm}-\smallint\limits_{G}K(2^{-j}\cdot (xz^{-1}))2^{-(j+k)Q}(\Phi(\mathcal{R})\delta)(2^{-j-k}\cdot z)dz|dx\\
&=   \smallint\limits_{|x|>2|y|}2^{-jQ} | \smallint\limits_{G}K((2^{-j}\cdot y^{-1})(2^{-j}\cdot x)(2^{-j}\cdot z^{-1})  ) 2^{-(j+k)Q}(\Phi(\mathcal{R})\delta)(2^{-j-k}z)dz\\
      &\hspace{4cm}-\smallint\limits_{G}K((2^{-j}\cdot x)(2^{-j}\cdot z^{-1}))2^{-(j+k)Q}(\Phi(\mathcal{R})\delta)(2^{-j-k}\cdot z)dz|dx\\
&=   \smallint\limits_{|x|>2|y|}2^{-jQ} | \smallint\limits_{G}K((2^{-j}\cdot y^{-1})(2^{-j}\cdot x)(2^{-j}\cdot z)^{-1}  ) 2^{-(j+k)Q}(\Phi(\mathcal{R})\delta)(2^{-j-k}z)dz\\
      &\hspace{4cm}-\smallint\limits_{G}K((2^{-j}\cdot x)(2^{-j}\cdot z)^{-1})2^{-(j+k)Q}(\Phi(\mathcal{R})\delta)(2^{-j-k}\cdot z)dz|dx      .
\end{align*}The change of variables $w:=2^{-j}\cdot z$ gives the new volume element $dw=2^{-jQ}dz,$ and we have  that
\begin{align*} &I_{j,k} \\
&=      \smallint\limits_{|x|>2|y|}2^{-jQ} | \smallint\limits_{G}K((2^{-j}\cdot y^{-1})(2^{-j}\cdot x)(2^{-j}\cdot z)^{-1}  ) 2^{-(j+k)Q}(\Phi(\mathcal{R})\delta)(2^{-j-k}z)dz\\
      &\hspace{4cm}-\smallint\limits_{G}K((2^{-j}\cdot x)(2^{-j}\cdot z)^{-1})2^{-(j+k)Q}(\Phi(\mathcal{R})\delta)(2^{-j-k}\cdot z)dz|dx
       \\
&=      \smallint\limits_{|x|>2|y|} | \smallint\limits_{G}K((2^{-j}\cdot y)^{-1}(2^{-j}\cdot x)w^{-1}  ) 2^{-kQ}2^{-jQ}(\Phi(\mathcal{R})\delta)(2^{-k}w)dw\\
      &\hspace{4cm}-\smallint\limits_{G}K((2^{-j}\cdot x)w^{-1})2^{-kQ}2^{-jQ}(\Phi(\mathcal{R})\delta)(2^{-k}\cdot w)dw|dx\\
 &=      \smallint\limits_{|x|>2|y|}2^{-jQ} | \smallint\limits_{G}K((2^{-j}\cdot y)^{-1}(2^{-j}\cdot x)w^{-1}  ) 2^{-kQ}(\Phi(\mathcal{R})\delta)(2^{-k}w)dw\\
      &\hspace{4cm}-\smallint\limits_{G}K((2^{-j}\cdot x)w^{-1})2^{-kQ}(\Phi(\mathcal{R})\delta)(2^{-k}\cdot w)dw|dx.
\end{align*}On the other hand, the change of variables $x'=2^{-j}\cdot x$ gives the new volume element $dx'=2^{-jQ}dx,$ the zone $\{x\in G:|x|>2|y|\}$ is mapped to the set 
$$   \{x'\in G: |x'|>2^{-j+1}|y|\}$$ allowing us to write the following identities
\begin{align*} &I_{j,k}\\
&=      \smallint\limits_{|x'|>2^{1-j}|y|} 2^{-jQ}  | \smallint\limits_{G}K((2^{-j}\cdot y)^{-1}{x'}w^{-1}  ) 2^{-kQ}(\Phi(\mathcal{R})\delta)(2^{-k}w)dw\\
      &\hspace{4cm}-\smallint\limits_{G}K({x'}w^{-1})2^{-kQ}(\Phi(\mathcal{R})\delta)(2^{-k}\cdot w)dw|2^{jQ}dx'\\
     &= \smallint\limits_{|x|>2^{1-j}|y|}  | \smallint\limits_{G}K((2^{-j}\cdot y)^{-1}{x}w^{-1}  )\Phi_k(w)dw-\smallint\limits_{G}K({x}w^{-1})\Phi_k(w)dw| dx\\
     &= \smallint\limits_{|x|>2^{1-j}|y|}  | K\ast \Phi_k((2^{-j}\cdot y)^{-1}{x})-K\ast \Phi_k({x})|dx\\
     &= \smallint\limits_{|x|>2^{1-j}|y|}  | \smallint\limits_{G}K(z)[\Phi_k(z^{-1}(2^{-j}\cdot y)^{-1}{x})-\Phi_k(z^{-1}{x})]dz|dx\\
     &= \smallint\limits_{|x|>2^{1-j}|y|} 2^{-kQ} | \smallint\limits_{G}K(z)[\Phi((2^{-k}\cdot z)^{-1}(2^{-j-k}\cdot y)^{-1}(2^{-k}\cdot{x}))-\Phi((2^{-k}\cdot z)^{-1}(2^{-k}\cdot {x}))]dz|dx\\
     &\leq  \smallint\limits_{|x|>2^{1-j}|y|} 2^{-k Q}  \smallint\limits_{G}|K(z)||[\Phi((2^{-k}\cdot z)^{-1}(2^{-j-k}\cdot y)^{-1}(2^{-k}\cdot{x}))-\Phi((2^{-k}\cdot z)^{-1}(2^{-k}\cdot {x}))]|dzdx,
\end{align*} where we have used the dilation property in \eqref{dilation:property} with $r=2^{-k}$.

The change of variables $x'=2^{-k}\cdot x$ and the new volume element $dx'=2^{-kQ}dx$ implies that
\begin{align*}
 &I_{j,k}\\
 & \leq \smallint\limits_{|x'|>2^{1-j-k}|y|}  \smallint\limits_{G}|K(z)||[\Phi((2^{-k}\cdot z)^{-1}(2^{-j-k}\cdot y)^{-1}(x'))-\Phi((2^{-k}\cdot z)^{-1}(x'))]|dzdx'.
\end{align*}Now, let us make  the change of variables $z''=2^{-k}\cdot z.$ We have that the new volume element  $2^{kQ}dz''=dz.$ Then, 
\begin{align*}
 &I_{j,k}\\
 & \leq \smallint\limits_{|x'|>2^{1-j-k}|y|}  \smallint\limits_{G}2^{kQ}|K(2^{k}\cdot z'')||[\Phi((z'')^{-1}(2^{-j-k}\cdot y)^{-1}(x'))-\Phi((z'')^{-1}(x'))]|dz''dx'.
\end{align*}
Now, the mean value theorem (see \cite[Page 119]{FischerRuzhanskyBook}) implies that
\begin{align*}
 &|[\Phi((z'')^{-1}(2^{-j-k}\cdot y)^{-1}(x'))-\Phi((z'')^{-1}(x'))]| \\
 &\lesssim\sum_{\ell=1}^{n}|(2^{-j-k}\cdot y)^{-1}|^{\nu_{\ell} }  \sup_{|z'|\lesssim |(2^{-j-k}\cdot y)^{-1}|  }|(X_{z',\ell} (\Phi((z'')^{-1}z'(x'))) |,
\end{align*}
where we have written that $|z'|\lesssim |(2^{-j-k}\cdot y)^{-1}|  $  to indicate that the inequality
\begin{align}\label{Taylor:constant}
 {|z'|\leq c |(2^{-j-k}\cdot y)^{-1}|= c 2^{-j-k}| y|  }  
\end{align}is valid for some universal  constant $c>1$ as in the mean value theorem of \cite[Page 119]{FischerRuzhanskyBook}.
Using that  $ \smallint_{G}2^{kQ}|K(2^{k}\cdot z^{''})|dz''=\Vert K\Vert_{L^1(G)},$ we can estimate $I_{j,k}$ as follows
\begin{align*}
    &I_{j,k}\\
   &\leq 2^{kQ} \smallint\limits_{G} |K(2^{k}\cdot z^{''})|dz'' \\
   &\times \sup_{z\in G}\smallint\limits_{|x'|>2^{1-j-k}|y|}     \left(\sum_{\ell=1}^{n}2^{-(j+k)\nu_{\ell}}| y |^{\nu_{\ell}}\sup_{|z'|\lesssim 2^{-j-k} | y|  }|X_{z,\ell}( \Phi ( z^{-1}  z'x') |\right)  dx'\\
   &=\|K\|_{L^1(G)} \\
   &\times\sum_{\ell=1}^{n}2^{-(j+k)\nu_{\ell}}| y |^{\nu_{\ell}} \sup_{z\in G}\smallint\limits_{|x'|>2^{1-j-k}|y|}     \sup_{|z'|\lesssim 2^{-j-k} | y|  }|X_{z,\ell}( \Phi ( z^{-1}  z'x') |dx'.
\end{align*}
Using the Sobolev embedding theorem on $G$ (see Theorem 4.4.25 of \cite[page 241]{FischerRuzhanskyBook}), we can estimate for $M_0>Q/2,$
\begin{align*}
 & \sup_{z\in G}\smallint\limits_{|x'|>2^{1-j-k}|y|}     \sup_{|z'|\lesssim 2^{-j-k} | y|  }|X_{z,\ell}( \Phi ( z^{-1}  z'x') |dx'\\
  &\lesssim \sup_{z\in G}\smallint\limits_{G}     \sup_{z'\in G }|X_{z,\ell}( \Phi ( z^{-1}  z'x') |dx'\\
 &\lesssim  \sup_{z\in G}\sum_{[\beta]\leq M_0 }\smallint\limits_{G} \Vert X_{z'}^\beta X_{z,\ell}( \Phi ( z^{-1}  z'x') \Vert_{L^2(G,dz')}dx'\\
 &\lesssim  \sup_{z\in G}\sum_{[\beta]\leq M_0 }\left(\smallint\limits_{G}(1+|x'|)^{2M_0} \Vert X_{z'}^\beta X_{z,\ell}( \Phi ( z^{-1}  z'x') \Vert_{L^2(G,dz')}^{2}dx'\right)^{\frac{1}{2}}\left(\smallint\limits_{G}(1+|x'|)^{-2M_0}dx'\right)^{\frac{1}{2}}\\
 &\lesssim  \sup_{z\in G}\sum_{[\beta]\leq M_0 }\left(\smallint\limits_{G}\smallint\limits_{G}(1+|x'|)^{2M_0} | X_{z'}^\beta X_{z,\ell}( \Phi ( z^{-1}  z'x') |^{2}dx'dz'\right)^{\frac{1}{2}}<\infty,
\end{align*}where the convergence of the last integral is justified by Hulanicki theorem, see \cite[page 251]{FischerRuzhanskyBook}.
 All the analysis above shows that for all $j,k\in \mathbb{Z},$ 
\begin{equation}\label{First:estimate:Ijk}
\forall y\in G: 0<|y|\leq 1,\,\,    I_{j,k}\lesssim  \sum_{\ell=1}^{n}2^{-(j+k)\nu_{\ell}}| y |^{\nu_{\ell}}\Vert K\Vert_{L^1(G)}.
\end{equation}
In terms of $R>0$ defined by
\begin{equation}\label{R:support}
    R=\inf\{R'>0:\textnormal{supp}(K)\subset B(e,R')\},
\end{equation}
  and of $0<|y|\leq 1,$ let us make a more precise estimate of $I_{j,k}$ in the case where $2^{-j}|y|\geq 2R.$ This will be used in further analysis.   To do this, let us come back to the estimate
\begin{align*}
  I_{j,k}\lesssim \sum_{\ell=1}^{n}2^{-(j+k)\nu_{\ell}}| y |^{\nu_{\ell}} & \smallint\limits_{|x'|>2^{1-j-k}|y|}  \smallint\limits_{G}   |K(z)| \sup_{|z'|\lesssim 2^{-j-k} | y|  }|X_{z,\ell}( \Phi ((2^{-k}\cdot z)^{-1}  z'x') | dz\, dx'\\
 &= \sum_{\ell=1}^{n}2^{-(j+k)\nu_{\ell}}| y |^{\nu_{\ell}}  \mathscr{I}_{j,k,\ell},
\end{align*}
where 
$$    \mathscr{I}_{j,k,\ell}=\smallint\limits_{|x'|>2^{1-j-k}|y|}  \smallint\limits_{G}   |K(z)| \sup_{|z'|\lesssim 2^{-j-k} | y|  }|X_{z,\ell}( \Phi ((2^{-k}\cdot z)^{-1}  z'x') | dz\, dx'.$$
To estimate this double integral let us split it as follows
\begin{align*}
    &\mathscr{I}_{j,k,\ell}=\smallint\limits_{|x'|>2^{1-j-k}|y|}  \smallint\limits_{G}   |K(z)| \sup_{|z'|\lesssim 2^{-j-k} | y|  }|X_{z,\ell}( \Phi ((2^{-k}\cdot z)^{-1}  z'x') |  dz\, dx'\\
    &  =\smallint\limits_{|x'|>2^{1-j-k}|y|}  \smallint\limits_{\{z:|x'|\geq 2^{-k+2}|z|\}}   |K(z)| \sup_{|z'|\lesssim 2^{-j-k} | y|  }|X_{z,\ell}( \Phi ((2^{-k}\cdot z)^{-1}  z'x') |  dz\, dx'\\
    &\hspace{2cm}+ \smallint\limits_{|x'|>2^{1-j-k}|y|}  \smallint\limits_{\{z:|x'|<2^{-k+2}|z|\}} |K(z)| \sup_{|z'|\lesssim 2^{-j-k} | y|  }|X_{z,\ell}( \Phi ((2^{-k}\cdot z)^{-1}  z'x') |  dz\, dx'\\
    &=\mathscr{I}_{j,k,\ell}^{I}+\mathscr{I}_{j,k,\ell}^{II}.
\end{align*}
Observe that for the integral
\begin{align*}
  \mathscr{I}_{j,k,\ell}^{I}=  \smallint\limits_{|x'|>2^{1-j-k}|y|}  \smallint\limits_{\{z:|x'|\geq 2^{-k+2}|z|\}}   |K(z)| \sup_{|z'|\lesssim 2^{-j-k} | y|  }|X_{z,\ell}( \Phi ((2^{-k}\cdot z)^{-1}  z'x') |  dz\, dx',
\end{align*} the integral with respect to $dz$ is computed on the zone $\{z:|x'|\geq 2^{-k+2}|z|\},$ where one has $-|x'|/4\leq - 2^{-k}|z|.$ On the other hand for the integral with respect to $dx',$ on the region $\{x':|x'|>2^{1-j-k}|y|\},$ one has that $|x'|/2>2^{-j-k}|y|.$ In consequence, for some constant $0<C<c,$ where $c$ is the universal constant in \eqref{Taylor:constant}, and then independent of $j$ and $k,$ one has that
\begin{align*}
    |(2^{-k}\cdot z)^{-1}  z'x'|\geq |x'|-|z'|-2^{-k}|z| 
    &\geq C(|x'|-2^{-j-k}|y|-\frac{1}{4}|x'|)\\
    &= C(|x'|/2-2^{-j-k}|y|+|x'|/2-\frac{1}{4}|x'|)\\
    &>  C(|x'|/2-\frac{1}{4}|x'|)\\
    &= C \frac{|x'|}{4}.
\end{align*}Then, estimating again $|X_{z,\ell}( \Phi ((2^{-k}\cdot z)^{-1}  z'x') |\leq C_{L}(1+|((2^{-k}\cdot z)^{-1}  z'x')|)^{-L},$ with $L$ to be determined later, we have
\begin{align*}
     &\smallint\limits_{|x'|>2^{1-j-k}|y|}  \smallint\limits_{\{z:|x'|\geq 2^{-k+2}|z|\}}   |K(z)| \sup_{|z'|\lesssim 2^{-j-k} | y|  }|X_{z,\ell}( \Phi ((2^{-k}\cdot z)^{-1}  z'x') | dx'\,dz\\
     &\leq \smallint\limits_{|x'|>2^{1-j-k}|y|}  \smallint\limits_{\{z:|x'|\geq 2^{-k+2}|z|\}}   |K(z)| \sup_{|z'|\lesssim 2^{-j-k} | y|  }C_{L}(1+|((2^{-k}\cdot z)^{-1}  z'x')|)^{-L} dx'\,dz\\
     &\leq \smallint\limits_{|x'|>2^{1-j-k}|y|}  \smallint\limits_{G}   |K(z)| dz C_{L}(1+\frac{1}{4}|x'|)^{-L} dx'.
\end{align*}Then, with $L=Q+1+2\nu_\ell,$ we have
\begin{align*}
    &\smallint\limits_{|x'|>2^{1-j-k}|y|}  \smallint\limits_{G}   |K(z)| dz C_{L}(1+\frac{1}{4}|x'|)^{-L} dx'\\
    &=\Vert K\Vert_{L^1(G)} \smallint\limits_{|x'|>2^{1-j-k}|y|}   C_{L}(1+\frac{1}{4}|x'|)^{-Q-1}(1+|x'|)^{-2\nu_\ell} dx'\\
     &\lesssim\Vert K\Vert_{L^1(G)}(2^{1-j-k}|y|)^{-2\nu_\ell} \smallint\limits_{|x'|>2^{1-j-k}|y|}   C_{L}(1+\frac{1}{4}|x'|)^{-Q-1} dx'\\
     &\leq\Vert K\Vert_{L^1(G)}(2^{1-j-k}|y|)^{-2\nu_\ell} \smallint\limits_{G}   C_{L}(1+\frac{1}{4}|x'|)^{-Q-1} dx'.
\end{align*}In consequence we have that
\begin{align*}
    \mathscr{I}_{j,k,\ell}^{I}\lesssim \Vert K\Vert_{L^1(G)}(2^{1-j-k}|y|)^{-2\nu_\ell} .
\end{align*}
Now, note that when $2^{-j} |y|\geq 2R,$ and for $|x'|<2^{-k+2}|z|,$ (and then $|x'|<2^{-k+2}R,$ since $z$ belongs to the support of $K$) we have that
\begin{equation}
 \forall 0<|y|\leq 1,\,\,2^{-j} |y|\geq 2R,\,   \{(x',z):|x'|>2^{1-j-k}|y|  \}\cap \{(x',z):|x'|<2^{-k+2}|z|  \}=\emptyset.
\end{equation} This implies that when $ 0<|y|\leq 1,\,\,2^{-j} |y|\geq 2R,$
\begin{equation}
 \mathscr{I}_{j,k,\ell}^{II}=    \smallint\limits_{|x'|>2^{1-j-k}|y|}  \smallint\limits_{\{z:|x'|<2^{-k+2}|z|\}} |K(z)| \sup_{|z'|\lesssim 2^{-j-k} | y|  }|X_{z,\ell}( \Phi ((2^{-k}\cdot z)^{-1}  z'x') | dx'\,dz=0.
\end{equation}

In summarising we have proved the following estimates

\begin{itemize}\label{Aux:1}
    \item $\forall 0<|y|\leq 1,\,\forall j,k\in \mathbb{Z}\,$  
    \begin{equation}
     I_{j,k}\lesssim\sum_{\ell=1}^{n}  2^{-(j+k)\nu_{\ell}}| y |^{\nu_{\ell}}\Vert K\Vert_{L^1(G)}.   
    \end{equation}Moreover, in view of \eqref{Trivial:estimate:for:} we have that 
    $\forall 0<|y|\leq 1,\,\forall j,k\in \mathbb{Z}\,$  
    \begin{equation}\label{Aux:genial}
     I_{j,k}\lesssim\sum_{\ell=1}^{n} \min\{1,\, 2^{-(j+k)\nu_{\ell}}| y |^{\nu_{\ell}}\}\Vert K\Vert_{L^1(G)}.   
    \end{equation}
    \item $\forall 0<|y|\leq 1,\,\,2^{-j} |y|\geq 2R,$ 
    \begin{equation}\label{Aux:2}
         I_{j,k}\lesssim \sum_{\ell=1}^{n} 2^{-(j+k)\nu_{\ell}}| y |^{\nu_{\ell}} (2^{1-j-k}|y|)^{-2\nu_\ell}\Vert K\Vert_{L^1(G)}\lesssim  \sum_{\ell=1}^{n} (2^{-(j+k)\nu_{\ell}}| y |^{\nu_{\ell}})^{-1}\Vert K\Vert_{L^1(G)}.
    \end{equation}
\end{itemize} Now, let us use these inequalities to estimate the H\"ormander condition
\begin{equation}    \label{auxiliar:weak:1:1:2}
   \sup_{|y|\leq 1} \smallint\limits_{|x|>2|y|}|\tilde{T}_{k}\delta(y^{-1}x)-\tilde{T}_{k}\delta(x)|dx\leq \sum_{j\in \mathbb{Z}}I_{j,k}.
\end{equation}Indeed, we consider the cases where $2^{-k}>(2R)^{-1}$ and where $2^{-k}\leq (2R)^{-1}.$
\begin{itemize}
    \item Case (i):  $2^{-k}>(2R)^{-1}.$ Observe that when $ \frac{1}{2^{-k}|y|}<2^{-j},$
    we have that $1< 2^{-j-k}|y|.$ In this situation, the fact that $\min\{1,2^{-j-k}|y|\}=1,$ and the estimate in \eqref{Aux:genial} imply that
    \begin{align*}
       \sum_{j}I_{j,k} & \lesssim\sum_{\ell=1}^n \Vert K\Vert_{L^1(G)}\left(\sum_{2^{-j}\leq \frac{1}{2^{-k}|y|}}2^{-(j+k)\nu_{\ell}}| y |^{\nu_{\ell}}+\sum_{\frac{1}{2^{-k}|y|}<  2^{-j}\leq \frac{2R}{|y|} } 1+\sum_{2^{-j}\geq \frac{2R}{|y|}  } (2^{-(j+k)\nu_{\ell}}| y |^{\nu_{\ell}})^{-1}\right)\\
       &\lesssim \sum_{\ell=1}^n \Vert K\Vert_{L^1(G)}(\log (R)+|k|)\lesssim  \Vert K\Vert_{L^1(G)}(\log (R)+|k|)\lesssim_{n,K}\Vert K\Vert_{L^1(G)}(1+|k|) .
    \end{align*} Indeed, observe that the sums $\sum_{2^{-j}\leq \frac{1}{2^{-k}|y|}}2^{-(j+k)\nu_{\ell}}| y |^{\nu_{\ell}}$ and $\sum_{2^{-j}\geq \frac{2R}{|y|}  } (2^{-(j+k)\nu_{\ell}}| y |^{\nu_{\ell}})^{-1}$ can be handled as geometric sums in order to get
    \begin{align*}
        \sum_{2^{-j}\leq \frac{1}{2^{-k}|y|}}2^{-(j+k)\nu_{\ell}}| y |^{\nu_{\ell}}=2^{-k\nu_\ell}|y|^{\nu_\ell}\sum_{2^{-j}\leq \frac{1}{2^{-k}|y|}}2^{-j\nu_{\ell}} \asymp 2^{-k\nu_\ell}|y|^{\nu_\ell}(2^{-k\nu_\ell}|y|^{\nu_\ell})^{-1}=1,
    \end{align*}and 
    \begin{align*}
        \sum_{2^{-j}\geq \frac{2R}{|y|}  } (2^{-(j+k)\nu_{\ell}}| y |^{\nu_{\ell}})^{-1}&=  (2^{-k\nu_{\ell}}| y |^{\nu_{\ell}})^{-1} \sum_{2^{j}\leq \frac{|y|}{2R}  } 2^{j\nu_{\ell}}\leq  (2^{-k\nu_{\ell}}| y |^{\nu_{\ell}})^{-1} \sum_{2^{j}\leq \frac{|y|}{2^{k}}  } 2^{j\nu_{\ell}}\\
        &\lesssim (2^{-k\nu_{\ell}}| y |^{\nu_{\ell}})^{-1}2^{-k\nu_{\ell}}| y |^{\nu_{\ell}}= 1. 
    \end{align*}As for the sum $\sum_{\frac{1}{2^{-k}|y|}<  2^{-j}\leq \frac{2R}{|y|} } 1,$ let $j_0\in \mathbb{Z}$ and $\ell_0\in \mathbb{N},$ be such that
    $$2^{-j_0}\leq \frac{1}{2^{-k}|y|}\leq 2^{-j_0+1},\,\, 2^{-j_0+\ell_0}\leq \frac{2R}{|y|}\leq 2^{-j_0+\ell_0+1} .  $$
    Then
    \begin{align*}
      \sum_{\frac{1}{2^{-k}|y|}<  2^{-j}\leq \frac{2R}{|y|} } 1\lesssim|\{j\in \{0,1,\cdots, \ell_0\}:2^{-j_0+j}\in [2^{-j_0},2^{-j_0+\ell_0}]\}|=\ell_0+1\lesssim\ell_0. 
    \end{align*}Now, to estimate $\ell_0$ in terms of $R$ and of $|k|,$ note that $2^{-j_0}\asymp  \frac{1}{2^{-k}|y|},$ and $2^{-j_0+\ell_0}\asymp   \frac{2R}{|y|}.$ We have that
    $$2^{-j_0-k}\asymp \frac{1}{|y|}\asymp \frac{2^{-j_0+\ell_0}}{2R}.   $$
    From this estimate we deduce that $2^{\ell_0}\asymp 2^{-k}\times 2R\lesssim  2^{|k|+\log_{2}(R)},$ from where we have that $\ell_0\lesssim  \textnormal{log}(R)+|k|\lesssim_{R}(1+|k|). $

  \item   Case (ii):  $2^{-k}\leq (2R)^{-1}.$ Again we divide the sum, but this time in two terms as follows
    \begin{align*}
        \sum_{j}I_{j,k} & \lesssim\sum_{\ell=1}^n \Vert K\Vert_{L^1(G)}\left(\sum_{2^{-j}\leq \frac{1}{2^{-k}|y|}} 2^{-(j+k)\nu_\ell}|y|^{\nu_\ell}+ \sum_{2^{-j}> \frac{1}{2^{-k}|y|}}(2^{-(j+k)\nu_\ell}|y|^{\nu_\ell})^{-1}\right)\\
         & \lesssim\sum_{\ell=1}^n \Vert K\Vert_{L^1(G)}=n\Vert K\Vert_{L^1(G)}.
    \end{align*}
\end{itemize}All the estimates above prove that 
\begin{equation}\label{auxiliar:weak:1:1:Final}
   \sup_{|y|\leq 1} \smallint\limits_{|x|>2|y|}|\tilde{T}_{k}\delta(y^{-1}x)-\tilde{T}_{k}\delta(x)|dx\lesssim_{n,K}\Vert K\Vert_{L^1(G)}(1+|k|),
\end{equation} as desired. So, for any $k\in \mathbb{Z}$ we have proved that 
$\tilde{T}_{k}:L^2(G)\rightarrow L^2(G)$ and $\tilde{T}_{k}:L^1(G)\rightarrow L^{1,\infty}(G)$ are bounded operators with the operator norms satisfying the bounds
\begin{equation}
    \Vert \tilde{T}_{k} \Vert_{\mathscr{B}(L^2(G))}\lesssim 2^{-a|k|},\,\, \Vert \tilde{T}_{k} \Vert_{\mathscr{B}(L^1(G),L^{1,\infty}(G))}\lesssim (1+|k|).
\end{equation} By using the Marcinkiewicz interpolation theorem, we have for all $1<p<2,$ the bound
\begin{equation}
   \Vert \tilde{T}_{k} \Vert_{\mathscr{B}(L^p(G))}\lesssim 2^{-a|k|\theta_p}(1+|k|)^{1-|\theta_p|},
\end{equation} where $1/p=\theta_p/2+(1-\theta_p).$ Consequently 
\begin{equation}
    \Vert T\Vert_{\mathscr{B}(L^p(G))}\lesssim \sum_{k\in \mathbb{Z}}  \Vert \tilde{T}_{k} \Vert_{\mathscr{B}(L^p(G))}\lesssim \sum_{k\in \mathbb{Z} }2^{-a|k|\theta_p}(1+|k|)^{1-|\theta_p|}<\infty,
\end{equation}for all $1<p<2.$ The boundedness of $T$ on $L^p(G)$ for all $2<p<\infty$ now follows from the duality argument.
\end{proof}
\subsection{The probabilistic argument}
Now we present a lemma about the $G$-function associated to the family $K_j$ of kernels in Lemma \ref{Lemma:main:dyadic}.
\begin{lemma}\label{Rademacher:lemma}  Let $K$ be a distribution  as in Lemma \ref{Lemma:main:dyadic}. Then the square function 
\begin{equation}
    (\mathscr{G}(K))f(x):=\left(\sum_{j\in \mathbb{Z}} |f\ast K_j(x)|^2  \right)^{\frac{1}{2}}
\end{equation}is a bounded operator from $L^p(G)$ to $L^p(G)$ for all $1<p<\infty.$    
\end{lemma}
\begin{proof}
    Let us give the classical probabilistic argument involving the Rademacher functions. So, given a sequence $\varepsilon=\{\varepsilon_j\}_{j\in \mathbb{Z}},$ where $\varepsilon_j=\pm 1,$ consider the operator
    \begin{equation}
        T_{\varepsilon}f(x)=\sum_{j\in \mathbb{Z}}\varepsilon_j f\ast K_j.
    \end{equation}In view of Lemma \ref{Lemma:main:dyadic}, $T_{\varepsilon}$ is bounded on $L^p(G)$ for all $1<p<\infty,$ and 
    \begin{equation}\label{choice}
        \Vert  T_\varepsilon\Vert_{\mathscr{B}(L^p)}\leq C_p,
    \end{equation} where $C_p$ is independent of the choice of the sequence $\varepsilon.$
 Let us consider the orthonormal system on $L^2[0,1]$ determined by the Rademacher functions $r_j,$ $j\in \mathbb{Z}.$ These are functions
defined via $r_j(t)=r_0(2^jt),$ where $r_0$ is defined via
\begin{equation}
    r_0(s)=-1,\, 0\leq s<1/2,\,\,r_0(s)=1,\, 1/2\leq s\leq 1.
\end{equation} The Rademacher functions $r_j,$ $j\in \mathbb{Z},$ satisfy the Khintchine inequality (see e.g. Grafakos \cite[Appendix C.2]{Grafakos}): 
\begin{itemize}
    \item If $F=\sum_j a_j r_j\in L^2[0,1],$ $\{a_j\in \ell^2\},$ then the Khintchine inequality
    \begin{equation}
    A_p    \Vert F\Vert_{L^p[0,1]}\leq \left(\sum_{j\in \mathbb{Z}} |a_j|^2  \right)^{\frac{1}{2}}\leq B_p \Vert F\Vert_{L^p[0,1]},
    \end{equation} holds for some $A_p,B_p>0.$
\end{itemize}Let $x\in G.$ If we apply this property with 
$ 
F(t)=(\mathscr{G}(K))f=\sum_{j\in \mathbb{Z}} f\ast K_j(x) r_j(t),
$ then
\begin{align*}
    \left(\sum_{j\in \mathbb{Z}} |f\ast K_j(x)|^2  \right)^{\frac{p}{2}}\leq B_p^p\int\limits_{[0,1]}\left|\sum_{j\in \mathbb{Z}} f\ast K_j(x) r_j(t)\right|^pdt. 
\end{align*}Now, integrating both sides of this inequality with respect to the Haar measure $dx,$ then for any $t\in [0,1],$ we have an operator like $T_\varepsilon=T_{ \{r_j(t)\}  }$ and the desired inequality
\begin{equation}
    \Vert (\mathscr{G}(K))f\Vert_{L^p(G)}\leq C_p \Vert f \Vert_{L^p(G)},
\end{equation}now follows from \eqref{choice}. The proof of Lemma \ref{Rademacher:lemma}  is complete.
\end{proof}

\subsection{Proof of the main theorem} Let $1<p\leq \infty.$
In this subsection, we prove the $L^p(G)$-boundedness of the dyadic maximal function \eqref{Maximal:Function:Graded} associated to a finite Borel measure $d\sigma$   with compact support on $G.$ We  assume that for some $a>0$ the group Fourier transform of $d\sigma$ satisfies the growth estimate in \eqref{FT;Condition:Measure}.

\begin{proof}[Proof of Theorem \ref{Main:Thn}] By the Riesz representation theorem we have that $d\sigma= Kdx,$ where $K$ is a distribution with compact support on $G$ that satisfies the group Fourier transform inequalities:
$
    \sup_{\pi\in\widehat{G}}\Vert\pi(\mathcal{R})^{\pm \frac{a}{\nu}}\widehat{K}(\pi)\Vert_{\textnormal{op}}<\infty.
$  Note that
\begin{align*}
     \mathcal{M}^{d\sigma}_Df(x) &=\sup_{j\in \mathbb{Z}}\left|\smallint\limits_{G}f\left(x\cdot (2^j\cdot y)^{-1}\right)d\sigma(y)\right|=\sup_{j\in \mathbb{Z}}\left|\smallint\limits_{G}f\left(x\cdot (2^j\cdot y)^{-1}\right)K(y)dy\right|\\
     &=\sup_{j\in \mathbb{Z}}\left|\smallint\limits_{G}f(xy^{-1})K_j(y)dy\right|, \, \forall y\in G,\,K_j(y):=2^{-jQ}K(2^{-j}y).
\end{align*}
Note that 
\begin{align*}
    f* {K}_j(x)=\smallint_Gf(xy^{-1})K_j(y)dy.
\end{align*}
In view of Lemma \ref{Rademacher:lemma}, we have that
\begin{equation}
 \mathcal{M}^{d\sigma}_Df(x)=\sup_{j\in \mathbb{Z}}\left|f\ast K_j(x)\right|\leq    (\mathscr{G}(K))f(x):=\left(\sum_{j\in \mathbb{Z}} |f\ast K_j(x)|^2  \right)^{\frac{1}{2}},
\end{equation}and in consequence we have proved the boundedness of $\mathcal{M}^{d\sigma}_D$ on $L^p(G)$ for all $1<p<\infty.$ The boundedness of $\mathcal{M}^{d\sigma}_D$ on $L^\infty(G)$ is clearly satisfied and in that case the $L^\infty(G)$-operator norm of $\mathcal{M}^{d\sigma}_D$ is bounded by the variation of the finite measure $d\sigma.$ The proof of Theorem \ref{Main:Thn} is complete.    
\end{proof}
\subsection{Conflict of interests statement.}   On behalf of all authors, the corresponding author states that there is no conflict of interest.

\subsection{Data Availability Statements.}  Data sharing does not apply to this article as no datasets were generated or
analysed during the current study.\\

\bibliographystyle{amsplain}

\begin{thebibliography}{99}



\bibitem{BagchiHaitRoncalThangavelu2018}   Bagchi, S.,  Hait, S., Roncal, L.  Thangavelu, S. On the maximal function associated to the spherical
means on the Heisenberg group,  New York J. Math. 27, 631--675, (2021).

\bibitem{BeltranGuoHickmanSeeger} Beltran, D. Guo, S, Hickman, J. Seeger, A. The circular maximal operator
on Heisenberg radial functions, to appear in Ann. Sc. Norm. Super. Pisa Cl. Sci.

\bibitem{Bourgain1986}  Bourgain, J. Averages in the plane over convex curves and maximal operators, J. Anal. Math. 47, 69--85, (1986).


\bibitem{Bourgain2020} Bourgain, J. Mirek, M. Stein, E. Wr\'obel, B. On discrete Hardy-Littlewood maximal functions over the balls in Zd: dimension-free estimates. Geometric aspects of functional analysis. Vol. I, 127–169, Lecture Notes in Math., 2256, Springer, Cham, 2020. 

\bibitem{Bourgain2021}  Bourgain, J. Mirek, M. Stein, E. Wr\'obel, B. On the Hardy-Littlewood maximal functions in high dimensions: continuous and discrete perspective. Geometric aspects of harmonic analysis, 107–148, Springer INdAM Ser., 45, Springer, Cham, 2021.



\bibitem{CalderonZygmund1952} Calder\'on, A. P., Zygmund, A. On the existence of certain singular integrals. Acta Math., 88, 85--139,
(1952). 

\bibitem{Calderon1979}   Calder\'on, C. P.  Lacunary spherical means, Ill. J. Math. 23, 476--484, (1979).

\bibitem{Carbery} Carbery, A. Radial Fourier multipliers and associated maximal functions, Recent progress in
Fourier analysis, North-Holland Mathematical Studies, 111 (North-Holland, Amsterdam, 1985)
49--56.

\bibitem{CladekKrause} Cladek, L. Krause, Ben.  Improved endpoint bounds for the lacunary spherical maximal operator,  	arXiv:1703.01508.

\bibitem{CoifmanWeiss1971} Coifman, R., Weiss, G. Analyse harmonique non commutative sur certains espaces homog\`enes. Lect. Notes Math. 242. (1971)


\bibitem{CoifmanWeiss1978} Coifman, R., Weiss, G. Review of the book Littlewood-Paley and multiplier theory. Bull. Am. Math. Soc. 84, 242--250, (1978).

 

\bibitem{CowlingMaceuri} Cowling, M. Mauceri, G. Inequalities for some maximal functions I. Trans. Amer. Math. Soc. 287, 431--455, (1985).



\bibitem{Cowling79}   Cowling,  M. On Littlewood-Paley-Stein theory, Rend. Semin. Mat. Fis. Milano XLIX, 79--87, (1979).

\bibitem{Christ88}  Christ, M. Weak type (1, 1) bounds for rough operators, Annals of Math. 128,
19--42, (1988).

\bibitem{ChristStein87} Christ, M.  Stein, E. A remark on singular Calder\'on-Zygmund theory, Proc. Amer. Math. Soc. 99, 71--75, (1987).


\bibitem{Duoandikoetxea2000} Duoandikoetxea, J. {Fourier Analysis. } {29}, American Mathematical Society, Providence (2000).

\bibitem{Fefferman1970} Fefferman, C. Inequalities for strongly singular integral operators, Acta Math. 24,  9--36, (1970).



\bibitem{Duoandikoetxea:RubioDeFrancia:86}  Duoandikoetxea, J.  Rubio de Francia, J. L. Maximal and singular integral operators
via Fourier transform estimates, Invent. Math. 84, 541--561, (1986).


\bibitem{FeffermanStein1972} Fefferman, C., Stein, E. $H^p$ spaces of several variables. Acta Math., 129, 137--193, (1972). 

\bibitem{Fefferman1973} Fefferman, C. $L^p$-bounds for pseudo-differential operators, Israel J. Math. 14, 413--417, (1973).

\bibitem{Fischer2006}  Fischer, V. The spherical maximal function on the free two-step nilpotent Lie group, Math. Scand. 99, 99--118, (2006).



\bibitem{FischerRuzhanskyBook} Fischer V., Ruzhansky M., Quantization on nilpotent Lie groups, Progress in Mathematics, Vol. 314, Birkhauser, 2016. xiii+557pp.


\bibitem{FollandStein1982} Folland, G. Stein, E. Hardy Spaces on Homogeneous Groups, Princeton University
Press, Princeton, N.J., 1982.

\bibitem{GangulyThangavelu2021} Ganguly, P. Thangavelu, S. On the lacunary spherical maximal function on the Heisenberg group. J. Funct. Anal. 280, no. 3, Paper No. 108832, 32 pp. (2021).

\bibitem{Grafakos}  Grafakos, L. Classical Fourier Analysis.vol. 249, 3rd edn. Springer, New York (2014).


\bibitem{GovidanAswinHickman} Govindan Sheri, A. Hickman, J. Wright, J. Lacunary maximal functions on homogeneous groups. J. Funct. Anal. 286(3), Paper No. 110250, (2024).

\bibitem{Hardy1913}   Hardy, G.  H.  A  theorem concerning  Taylor's series,  Quart. J.  Pure  Appl. Math., 44,  147--160, (1913).



\bibitem{HelfferNourrigat}  Helffer, B.  Nourrigat, J. Caracterisation des operateurs hypoelliptiques homogenes invariants a gauche sur un groupe de Lie nilpotent
gradue, Comm. Partial Differential Equations, 4(8), 899--958, (1979).


\bibitem{Hormander1960} H\"ormander, L. Estimates for translation invariant operators in Lp spaces,
Acta Math., 104, 93--139, (1960).



\bibitem{Guorong:Ruz} Hong Q., Hu G., Ruzhansky M., Fourier multipliers for Hardy spaces on graded Lie groups. To appear in Proc. Royal Soc. Edinburgh. arXiv:2101.07499. 

\bibitem{Lacey2019} Lacey, M. T. Sparse bounds for spherical maximal functions, J. Anal. Math. 139, (2019), 613--635.

\bibitem{Mockenhaupt1992} Mockenhaupt, G. Seeger, A., Sogge, C. Wave front sets, local smoothing and Bourgain's
circular maximal theorem, Ann. of Math. 136, (1992), 207--218.

\bibitem{MullerSeeger2004} M\"uller, D.,  Seeger, A. Singular spherical maximal operators on a class of two-step nilpotent Lie groups, Isr. J. Math. 141, 315--340, (2004).


\bibitem{Nagel1976} Nagel, A.  Rivi\`ere, N. M., Wainger, S. On Hilbert transforms along curves, II, Amer. J. Math. 98, 395--403, (1976).

\bibitem{NarayananThangavelu2004}  Narayanan, E. K.,  Thangavelu, S. An optimal theorem for the spherical maximal operator on the
Heisenberg group, Isr. J. Math. 144,  211--219, (2004).

\bibitem{NevoThangavelu}   Nevo, A.  Thangavelu, S. Pointwise ergodic theorems for radial averages on the Heisenberg group, Adv. Math. 127,  307--339, (1997).

\bibitem{RoosSeegerSrivastava22} Roos, J. Seeger, A. Srivastava, R. Lebesgue space estimates for spherical maximal functions on Heisenberg groups. Int. Math. Res. Not. IMRN, no. 24, 19222--19257, (2022). 

\bibitem{RothschildStein76}  Rothschild, L. P.,  Stein, E. M. Hypoelliptic differential operators
and nilpotent groups. Acta Math., 137(3--4), 247--320, (1976).

\bibitem{RubioDeFrancia}  Rubio De Francia, J. L. Maximal functions and Fourier transforms, Duke Math. J. 53, 395--404, (1986).


\bibitem{Seeger} Seeger, A. Some inequalities for singular convolution operators in Lp-spaces, Trans. Amer. Math. Soc.,
308(1), 259--272, (1988).

\bibitem{Seeger1990} Seeger, A. Remarks on singular convolution operators. Studia Math. 97(2), 91--114,  (1990).

\bibitem{Seeger1991} Seeger, A. Endpoint estimates for multiplier transformations on compact manifolds. Indiana Univ. Math. J. 40(2), 471--533, (1991).

\bibitem{SeegerSogge}  Seeger, A. Sogge, C. D. On the boundedness of functions of (pseudo-) differential operators on compact manifolds. Duke Math. J. 59(3), 709--736, (1989).

\bibitem{SSS}
Seeger, A.,  Sogge, C. D.,   Stein, E. M.
\newblock Regularity properties of {F}ourier integral operators.
\newblock { Ann. of Math.}, 134(2), 231--251, (1991).

\bibitem{STW}  Seeger, A.  Tao, T.  Wright, T. Singular maximal functions and Radon transforms
near $L^1$, Amer. J. Math. 126(3), 607--647, (2004).




\bibitem{Schmidt}  Schmidt, O. Maximaloperatoren zu Hyperfl\"achen in Gruppen vom homogenen Typ, Diplomarbeit
an der Christian-Albrechts-Universit\"at zu Kiel, Mai 1998.

\bibitem{Stein1976} Stein, E. M. Maximal functions. I. Spherical means, Proc. Natl. Acad. Sci. USA 73, 2174--2175, (1976).



\bibitem{Stein1970} Stein, E. M., Singular integrals and differentiability properties of functions. Princeton 1970. 


\bibitem{Taylorbook1981} Taylor, M.  Pseudodifferential Operators, Princeton Univ. Press, Princeton, N.J., 1981.

\bibitem{SeegerTaoWright} Seeger, A.  Tao, T.  Wright, J. Singular maximal functions and Radon transforms near $L^1$,
Amer. J. Math. 126, 607--647, (2004).

\bibitem{Wainger1965}  Wainger, S.  Special trigonometric series in $k$-dimensions, Mem. Amer. Math. Soc., 59, (1965).


\end{thebibliography}

\end{document}